\documentclass[12pt]{amsart}
\usepackage{graphicx} 
\usepackage{csquotes}
\usepackage[colorlinks=true]{hyperref} 
\usepackage[backend=biber,style=alphabetic,url=false,doi=false,isbn=false,sorting=nyt,maxbibnames=9,giveninits=true]{biblatex}
\usepackage[a4paper]{geometry}
\usepackage{amsfonts}
\usepackage{amsthm}
\usepackage{amsmath}
\usepackage{parskip}
\usepackage{mathrsfs}
\usepackage{graphicx}
\usepackage{amssymb}
\usepackage{xtab}
\usepackage{enumitem}
\usepackage{tikz-cd}
\usepackage{mathtools}
\usepackage{comment}
\usepackage{stmaryrd}

\usepackage{todonotes}

 \newcommand{\Z}{\ensuremath{\mathbb{Z}}}
 \newcommand{\N}{\ensuremath{\mathbb{N}}}
 \newcommand{\R}{\ensuremath{\mathbb{R}}}
 \newcommand{\C}{\ensuremath{\mathbb{C}}}
 \newcommand{\Q}{\ensuremath{\mathbb{Q}}}
 \newcommand{\U}{\ensuremath{\mathbb{U}}}
 \newcommand{\A}{\ensuremath{\mathbb{A}}}
 \newcommand{\bV}{\ensuremath{\mathbb{V}}}
 
 \newcommand{\M}{\ensuremath{\mathcal{M}}}
 \newcommand{\PP}{\ensuremath{\mathbb{P}}}
 
 \newcommand{\X}{\ensuremath{\mathcal{X}}}
 \newcommand{\OO}{\ensuremath{\mathcal{O}}}
 \newcommand{\cX}{\ensuremath{\mathscr{X}}}
 \newcommand{\cH}{\ensuremath{\mathscr{H}}}
 \newcommand{\cY}{\ensuremath{\mathscr{Y}}}
 \newcommand{\cV}{\ensuremath{\mathscr{V}}}
 \newcommand{\cC}{\ensuremath{\mathscr{C}}}
 \newcommand{\cE}{\ensuremath{\mathscr{E}}}
 \newcommand{\cL}{\ensuremath{\mathscr{L}}}
 \newcommand{\cF}{\ensuremath{\mathscr{F}}}
 
\newcommand{\wcX}{\ensuremath{\widetilde{\mathscr{X}}}}
\newcommand{\wpi}{\ensuremath{\widetilde{\pi}}}
\newcommand{\wcL}{\ensuremath{\widetilde{\mathscr{L}}}}

\newcommand{\wcH}{\ensuremath{\widetilde{\mathscr{H}}}}

\newcommand{\wX}{\ensuremath{\widetilde{X}}}
\newcommand{\wL}{\ensuremath{\widetilde{L}}}

\newcommand{\wS}{\ensuremath{\widetilde{S}}}
\newcommand{\wC}{\ensuremath{\widetilde{C}}}

\newcommand{\wH}{\ensuremath{\widetilde{H}}}

\usepackage{setspace}

\DeclareMathOperator{\Id}{Id}

\DeclareMathOperator{\ord}{ord}
\DeclareMathOperator{\vol}{vol}
\DeclareMathOperator{\Aut}{Aut}

\DeclareMathOperator{\Pic}{Pic}
\DeclareMathOperator{\Bs}{Bs}

\DeclareMathOperator{\Bl}{Bl}
\DeclareMathOperator{\Spec}{Spec}
\DeclareMathOperator{\coeff}{coeff}

\DeclareMathOperator{\SL}{SL}

\DeclareMathOperator{\PGL}{PGL}
\DeclareMathOperator{\CM}{CM}
\DeclareMathOperator{\Proj}{Proj}
\DeclareMathOperator{\NS}{NS}

\DeclareMathOperator{\Sym}{Sym}
\DeclareMathOperator{\rank}{rank}

\newcommand{\wHthree}{\widetilde{H}_3}

\newcommand{\wLprime}{\widetilde{L}'}

\newtheorem{prop}{Proposition}[section]
\newtheorem{lemma}[prop]{Lemma}

\newtheorem{theorem}[prop]{Theorem}

\newtheorem{cor}[prop]{Corollary}

\newtheorem{remark}[prop]{Remark}

\newtheorem{defn}[prop]{Definition}

\title{K-Moduli of Fano Threefolds of Family 3.3}

\author[E. Etxabarri-Alberdi]{Erroxe Etxabarri-Alberdi}
\address{ Basic Sciences department, Mondragon Unibertsitatea, Faculty of Engineering, Loramendi Kalea, 4, 20500 Arrasate, Gipuzkoa, Spain}
\email{eetxabarri@mondragon.edu}

\author[J. M. Jones]{James Matthew Jones}
\address{FB12 / Mathematik und Informatik, Philipps-Universität Marburg, Hans-Meerwein-Str. 6, 35032 Marburg, Germany}
\email{jonesj@mathematik.uni-marburg.de}

\author[T. S. Papazachariou]{Theodoros Stylianos Papazachariou}
\address{Yau Mathematical Sciences Center, Jingzhai, Tsinghua University, Haidian District, Beijing, China, 100084}
\email{tpapazachariou@mail.tsinghua.ed.cn}

\begin{document}

\begin{abstract}
    We explicitly fully describe the K-moduli space of Fano threefold family \textnumero3.3. We first show that K-semistable Fano varieties with volume greater than 18 are Gorenstein canonical and admit general elephants, decreasing the bound on a result by Liu and Zhao. Combining this with the moduli-continuity method via lattice-polarized K3 surfaces, we identify the K-moduli stack parametrising K-semistable varieties in family \textnumero3.3 with a Kirwan blow up of the natural GIT quotient of $(1,1,2)$ divisors in $\PP^1\times \PP^1\times \PP^2$.
\end{abstract}
\maketitle

\section{Introduction}
One of the key goals in modern algebraic geometry is the construction of moduli spaces parametrising varieties up to isomorphism. In recent years, moduli spaces were constructed for Fano varieties via K-stability, a very important subclass of varieties due to their link with the Minimal Model Program (MMP) and mirror symmetry. K-stability is an algebrogeometric notion developed to determine the existence of K{\"a}hler--Einstein (KE) metrics on Fano varieties, and has made great achievements in the construction of K-moduli spaces of such varieties and log Fano pairs. The general K-moduli theorem (cf. Theorem \ref{K-moduli theorem}) is a combination of substantial work by a number of people (cf. \cite{ABHLX20,BHLLX21,BLX19,BX19,CP21,Jia20,LWX21,LXZ22,Xu20,XZ20, OSS16, SSY, XZ21}), which establishes that for fixed dimension $n$ and log-anticanonical volume $v$, the functor of K-semistable log Fano pairs is represented by a separated Artin stack, which admits a projective good moduli space (in the sense of \cite{Alp13});  the closed points of the moduli space are in bijection with isomorphic classes of  K-polystable log Fano pairs.

Despite the monumental achievement already mentioned, the construction of K-moduli spaces is not explicit. It varies from example to example, making the explicit description of K-moduli spaces a difficult problem. As a natural next step, obtaining the full description of all K-(semi/poly)stable degenerations for specific families of Fano varieties has become an integral research topic for algebraic geometers, with exponential success in recent years. 

Most examples use the moduli continuity method, or variants of it, to relate K-moduli spaces to GIT quotients or their modifications. This was initiated from the study of moduli spaces of KE del Pezzo surfaces \cite{mabuchi_mukai_1990, OSS16}, and expanded to Fano threefolds and higher dimensions \cite{spotti_sun_2017, LX19,ADL19, Liu22, ADL19}. Another approach to describe K-moduli of Fano threefolds relies on stability threshold estimates via the Abban--Zhuang method \cite{AZ20}, which allows the immediate detection of singular K-polystable elements 
\cite{Pap22, Abban_Cheltsov_Denisova_Etxabarri-Alberdi_Kaloghiros_Jiao_Martinez-Garcia_Papazachariou_2024, cheltsov2024kmodulifanothreefoldsfamily, abban2024kmoduliquarticthreefolds, devleming2024kmodulispacefamilyconic,CheltsovEtAl2025}. Other methods include a comparative description of K-moduli spaces via products \cite{Pap24}. Most of the above results rely on the existence of general smooth K-stable members, which in the case of Fano threefolds has been mainly achieved by \cite{acc2021}. In recent works \cite{Liu22,LZ24,Zha24,DP25, TP25}, the geometry and moduli space of K3 surfaces are widely used, since they appear as sections in the anticanonical linear series of Fano threefolds. In particular, in \cite{LZ24} it is shown that a K3 surface always exists in the anticanonical linear system of K-semistable singular Fano threefolds of volume $\geq 20$.

In this paper, we study the K-stability and the K-moduli stack (and space) of smooth Fano threefolds in Family \textnumero3.3. Geometrically, these varieties are $(1,1,2)$ divisors in $\PP^1\times \PP^1\times \PP^2$. The volume of such varieties is $18$, hence the analysis in \cite{LZ24} cannot be directly applied. To that end, our first main result is more general, and it shows that the bound in \cite[Theorem 4.2]{LZ24} can be reduced to $18$ (or 16 with some extra conditions). 

\begin{theorem}[See Theorem \ref{gor canonical K-ss limits}]\label{intro:thm_vol}
    Let $X$ be a $\Q$-Gorenstein smoothable K-semistable (weak) $\Q$-Fano threefold with volume $V:=(-K_X)^3\geq 16$. Then the following hold:
    \begin{itemize}
        \item[(1)] The variety $X$ is Gorenstein canonical if $V \geq 18$, and klt with only Gorenstein canonical or $1/4(1,1,0)$ quotient singularities if $V= 16$.
        \item[(2)] If $V \geq 18$, there exists a divisor $S\in \vert -K_X\vert$ such that $(X,S)$ is a plt pair, and such that $(S,-K_X\vert_S)$ is a (quasi-) polarised K3 surface of degree $V$.
        \item[(3)] If $D$ is a $\Q$-Cartier Weil divisor on $X$ which deforms to a $\Q$-Cartier Weil divisor on a $\Q$-Gorenstein smoothing of $X$, then $D$ is Cartier.
    \end{itemize}
\end{theorem}

The proof of the above is more involved than the similar statement in \cite{LZ24}, and requires estimates established in \cite{LX19}. Using this, we can determine that each member in this K-moduli space is either a $(1,1,2)$ divisor in $\PP^1\times \PP^1\times \PP^2$ or a $(2,2)$ divisor in $\PP(1,1,2)\times \PP^2$. 

\begin{theorem}[See Corollary \ref{cor:main}]\label{intro:thm1}
    Let $\M^K_{\textup{№3.3}}$ be the K-moduli stack of the family \textnumero3.3, admitting good moduli space $M^K_{\textup{№3.3}}$. Then,
    \begin{enumerate}
        \item Every K-semistable Fano variety $X$ in $\M^K_{\textup{№3.3}}$ is isomorphic to either a $(1,1,2)$ divisor in $\PP^1\times \PP^1\times \PP^2$ or a $(2,2)$ divisor in $\PP(1,1,2)\times \PP^2$, and
        \item The K-moduli stack $\M^K_{\textup{№3.3}}$ is a smooth connected component of $\M^K_{3,18}$, and the K-moduli space $M^K_{\textup{№3.3}}$ is a normal projective variety.
    \end{enumerate}
\end{theorem}

Using the above result, we relate the moduli stack $\M^K_{\textup{№3.3}}$ to a natural GIT compactification.

Let $W= \PP(H^0(\PP^1\times \PP^1\times \PP^2,\mathcal{O}(1,1,2)))$ be the parameter space of $(1,1,2)$ divisors $X$ in $\PP^1\times \PP^1\times \PP^2$. We consider the associated GIT quotient
$$M^{GIT}_{1,1,2} \coloneqq W\sslash\PGL(2)\times \PGL(2)\times \PGL(3)$$
which yields a natural projective birational model for the K-moduli space $M^K_{\textup{№3.5}}$.

Let $\widetilde{X}$ be the non-reduced, reducible $(1,1,2)$ divisor, where $X = \mathbb{V}(\tilde{f})$, with 
$$\tilde{f} = (z_1^2+z_0z_2)(x_0y_0+x_1y_0+x_0y_1+x_1y_1).$$
Consider the Kirwan blow-up $\widetilde{\mathcal{U}}^{ss}_{1,1,2}$ of the GIT moduli stack $\M^{GIT}_{1,1,2}$ at the point parametrising $\widetilde{X}$. Our second result identifies the K-moduli space (resp. stack) with this Kirwan blow up.

\begin{theorem}\label{intro: thm 2}
   There is a natural isomorphism 
    $$\M^K_{\textup{№3.3}}\cong \left[\widetilde{\mathcal{U}}^{ss}_{1,1,2}/\PGL(2)\times \PGL(2)\times \PGL(3)\right],$$
    which descends to a canonical isomorphism of the good moduli spaces:
 $$M^K_{\textup{№3.3}}\ \simeq \ \widetilde{\mathcal{U}}^{ss}_{1,1,2}\sslash\PGL(2)\times \PGL(2)\times \PGL(3),$$
\end{theorem}

In addition, we fully describe the GIT semistable, polystable and stable orbits of the GIT quotient $M^{GIT}_{1,1,2}$, using computational methods from \cite{karagiorgis2023gitstabilitydivisorsproducts}. We also study the exceptional divisor of the Kirwan blow up $\widetilde{\mathcal{U}}^{ss}_{1,1,2}$, and its stable, polystable and semistable orbits using \cite[\S 3]{devleming2024kmodulispacefamilyconic}. As an immediate consequence of the above, we are able to obtain a full explicit description of all K-(semi/poly)stable elements in family \textnumero3.3.

\begin{theorem}[See Theorem \ref{thm: full K-ss description}]\label{intro:thm_description of K-st objects}
    A Fano threefold in family \textnumero3.3 is
    \begin{enumerate}
        \item K-stable if and only if it is smooth;
        \item strictly K-semistable if and only if $X$ is either a $(1,1,2)$ divisor in $\PP^1\times \PP^1\times\PP^2$ with either
        \begin{enumerate}
            \item non-isolated singularities of multiplicity $2$, or
       \item twelve $A_1$ singulatities, or
       \item one $A_3$ singularity, or
       \item one $A_3$ and one $A_1$ singularity, or
       \item one $D_4$ singularity.
        \end{enumerate}
        or a $(2,2)$ divisor in $\PP(1,1,2)_{u,v,s}\times\PP^2_{\mathbf{w}}$ in which case $X = V(f)$, where $f = s(w_0w_1+w_2^2)+g_{2,2}(u,v,\mathbf{w})$, such that $Y = V(g_{2,2})$ is a $(2,2)$ divisor in $\PP^1\times \PP^2$ which has 
        \begin{enumerate}
           \item non isolated singularities along the point at infinity of $\PP^1$ and is reducible, or
           \item two $A_1$ singularities, or
           \item one $A_2$ singularity, or 
           \item one $A_3$ singularity;
        \end{enumerate}
        \item strictly K-polystable if and only
        if $X$ is either a $(1,1,2)$ divisor in $\PP^1\times \PP^1\times\PP^2$ with either
        \begin{enumerate}
           \item two non-isolated singularities of multiplicity $2$, or
       \item eight $A_1$ singularities, or
       \item two $A_3$ singularities, or
       \item two $A_3$ and two $A_1$ singularities, or
       \item two $D_4$ singularities.
        \end{enumerate}
        or a $(2,2)$ divisor in $\PP(1,1,2)_{u,v,s}\times\PP^2_{\mathbf{w}}$ in which case $X = V(f)$, where $f = s(w_0w_1+w_2^2)+g_{2,2}(u,v,\mathbf{w})$, such that $Y = V(g_{2,2})$ is a $(2,2)$ divisor in $\PP^1\times \PP^2$ which has  
        
        \
        \begin{enumerate}
           \item non isolated singularities along along the points at zero and at infinity of $\PP^1$ and is reducible, or
           \item four $A_1$ singularities, or
           \item two $A_2$ singularities, or 
           \item two $A_3$ singularities.
        \end{enumerate}
    \end{enumerate}
\end{theorem}
\subsection*{Structure of the paper}
In Section \ref{sec:preliminaries} we recover some of the basic definitions on K-stability and K-moduli, as well as the definitions for K3 surfaces, and the moduli spaces of K3 surfaces. In Section \ref{sec: K-ss limits} we prove Theorem \ref{intro:thm_vol}, and we show that every K-semistable limit in family \textnumero3.3 is isomorphic to either a $(1,1,2)$ divisor in $\PP^1\times \PP^1\times \PP^2$ or a $(2,2)$ divisor in $\PP(1,1,2)\times \PP^2$, proving the first part of Theorem \ref{intro:thm1}. In Section \ref{sec:k-moduli} we study the GIT quotient of $(1,1,2)$ divisors in $\PP^1\times \PP^1\times \PP^2$, and we construct the Kirwan blow up $\widetilde{\mathcal{U}}^{ss}_{1,1,2}$ via a Luna slice. We furthermore show that GIT (semi/poly)stability corresponds to K-stability of elements in family \textnumero3.3 (after taking the Kirwan blow up), effectively proving Theorem \ref{intro: thm 2}. We also prove the second part of Theorem \ref{intro: thm 2} by studying the deformation theory of $(1,1,2)$ divisors in $\PP^1\times \PP^1\times \PP^2$ and $(2,2)$ divisors in $\PP(1,1,2)\times \PP^2$.

\subsection*{Acknowledgments}
We would like to thank Hamid Abban, Ivan Cheltsov, Elena Denisova, Yuchen Liu and Junyan Zhao for helpful discussions and useful comments. TSP is supported by  Beijing Natural Science Foundation Project IS25037 and a Shuimu Scholar Programme fellowship by Tsinghua University.

\section{Preliminaries}\label{sec:preliminaries}
\subsection{K-stability and K-moduli spaces}
In this section, we follow conventions and notation from  \cite{acc2021}.
\begin{defn}
A normal projective variety $X$ is called a $\Q$-\textup{ Fano variety} if $-K_X$ is an ample $\Q$-Cartier divisor and if $X$ is a Kawamata log terminal (\textup{klt}) variety.
\end{defn}

Now we define  K-stability for  Fano varieties, by first introducing the following birational invariants.

\begin{defn}
Let $X$ be an $n$-dimensional Fano variety, and $E$ a prime divisor on a normal projective variety $Y$, where $\pi:Y\rightarrow X$ is a birational morphism. Then the \textup{log discrepancy} of $X$ with respect to $E$ is $$A_{X}(E):=1+\coeff_{E}(K_Y-\pi^{*}(K_X)).$$ 
The \textup{expected vanishing order} of $X$ with respect to $E$ (also referred to as the \emph{$S$-invariant}) is $$S_{X}(E):= \frac{1}{(-K_X)^n}\int_0^{\tau}\vol(\pi^*(-K_X)-xE)dx,$$ 
(see the definition of $\vol$ in \cite[\S 2.2. C]{Laz04}).

The \textup{$\beta$-invariant} of $X$ with respect to $E$ is then 
$$\beta_{X}(E):= A_{X}(E)-S_{X}(E).$$
\end{defn}

\begin{theorem}[cf. {\textup{\cite{Fuj19b,Li17,BX19}}}]
    The following assertions hold:
    \begin{itemize}
        \item $X$ is K-stable if and only if $\beta_{X}(E)>0$ for every prime divisor $E$ over $X$;
        \item $X$ is K-semistable if and only if $\beta_{X}(E)\geq0$ for every prime divisor $E$ over $X$.
    \end{itemize}
\end{theorem}

\begin{defn}\label{Q-Gor K-ss Fano family}
    A $\Q$\emph{-Fano family} is a morphism $f\colon \X\rightarrow B$ of schemes such that 
    \begin{enumerate}
        \item $f$ is projective and flat of pure relative dimension $n\in \N_{>0}$;
        \item the geometric fibres of $f$ are $\Q$-Fano varieties;
        \item $-K_{\X/B}$ is $\Q$-Cartier and $f$-ample;
        \item $f$ satisfies Koll{\'a}r’s condition (see, {\cite[p. 24]{Kol08}}).
    \end{enumerate}
    We call a $\Q${-Fano family} $f\colon \X\rightarrow B$ a K-semistable $\Q${-Fano family} if in addition to the above, the fibres $\X_b$ are K-semistable for all $b\in B$.
\end{defn}

We are now in a position to define the K-moduli functor/stack for log Fano pairs. The existence and properness of the K-moduli functor as an Artin stack are due to a number of recent contributions \cite{Jia18, CP21, BHLLX21, Xu20Tow,BLX19, XZ21, ABHLX20, LXZ22}.

We begin with a definition.

\begin{defn}\label{k-moduli stack def}
    Let $n$, $V$ be positive integers. The \emph{K-moduli stack of K-semistable Fano varieties of dimension $n$ and volume $V$} is the Artin stack $\M^K_{n,V}$ defined as the functor sending a reduced base $S$ to 
    \[
{\M}^K_{n,V}(S):=\left\{\X\rightarrow S\left| \begin{array}{l}\X\rightarrow S\textrm{ is a K-semistable $\Q$-Fano family, where}\\
\textrm{the fibres have dimension $n$ and volume $V$}\\ 
\end{array}\right.\right\}.
\]
\end{defn}

\begin{theorem}[{{K-moduli Theorem}}]\label{K-moduli theorem} 
Let $\chi$ be the Hilbert polynomial of an anti-canonically polarised  Fano variety. Consider the moduli pseudo-functor sending a reduced base $S$ to

\[
{\M}^K_{X}(S):=\left\{\X/S\left| \begin{array}{l}\X /S\textrm{ is a $\Q$-Fano family,}\\ \textrm{each fibre $\X_s$ is K-semistable, and }\\ \textrm{$\chi(\X_s,\OO_{\X_s}(-mK_{\X_s}))=\chi(m)$ for $m$ sufficiently divisible.}\end{array}\right.\right\}.
\]
Then there is a reduced Artin stack ${\M}^K$ of finite type over $\C$ representing this moduli pseudo-functor. The $\C$-points of ${\M}^K$ parameterise K-semistable $\Q$-Fano varieties $X$ with Hilbert polynomial $\chi(X,\OO_X(-mK_X))=\chi(m)$ for sufficiently divisible $m\gg 0$. Moreover, the stack ${\M}^K$ admits a good moduli space ${M}^K$, which is a reduced projective scheme of finite type over $\C$, whose $\C$-points parameterise K-polystable Fano varieties. 
\end{theorem}

Here, pseudo-functor refers to a 2-functor which preserves composition and identities of 1-morphisms only up to coherent specified 2-isomorphism.

\subsection{Geometry of Fano threefolds in family \textnumero 3.3}\label{sec: fano geometry}

In this section, we review facts about the geometry of smooth Fano threefolds in family \textnumero3.3. For more details, we prompt the reader to \cite[\S 1]{CHELTSOV_FUJITA_KISHIMOTO_OKADA_2023}. Let $X$ be a smooth $(1,1,2)$ divisor in $\PP^1_{s,t}\times \PP^1_{u,v} \times \PP^2_{x,y,z}$ and let $\pi_i$ be the projections to each projective space for $i\in\{1,2,3\}$. Let $H_1=\pi_1^*(\OO_{\PP^1}(1))$, $H_2=\pi_2^*(\OO_{\PP^1}(1))$ and $H_3=\pi_1^*(\OO_{\PP^2}(1))$ be the natural line pullbacks of line bundles under each natural projection. We have $-K_X\sim H_1+H_2+H_3$, and $(-K_X)^3=18$. We also have a commutative diagram,

\[
\begin{tikzcd}
& \PP^1_{s,t} \times \PP^1_{u,v} \arrow[dr] \arrow[dl] & \\
\PP^1_{s,t} & & \PP^1_{u,v} \\
& X \arrow[uu, "\omega" near start]\arrow[ul, "\pi_1"'] \arrow[ur, "\pi_2"] \arrow[dl, "\phi_1"'] \arrow[dr, "\phi_2"] \arrow[dd, "\pi_3" near start] & \\
\PP^1_{s,t} \times \PP^2_{x,y,z}\arrow[uu]\arrow[dr] & & \PP^1_{u,v} \times \PP^2_{x,y,z} \arrow[uu]\arrow[dl]\\
& \PP^2_{x,y,z} &
\end{tikzcd}
\]

where:
\begin{enumerate}
    \item $\pi_1$ and $\pi_2$ are morphisms from $X$ to $\PP^1$ which are fibrations into quintic del Pezzo surfaces, induced by $|H_1|$ and $|H_2|$, respectively.
    \item $\pi_3$ is a conic bundle induced by $|H_3|$ whose discriminant curve is a smooth plane quartic curve.
    \item $\omega$ is a conic bundle induced by $|H_1+H_2|$ whose discriminant curve is a curve of degree $(3,3)$.
    \item $\phi_1$ and $\phi_2$ are blow ups of  $\PP^1\times \PP^2$ along a genus $3$ curve given as the complete intersection of two $(1,2)$ surfaces in $\PP^1\times \PP^2$. These are induced by $|H_1+H_3|$ and $|H_2+H_3|$ respectively, and admit exceptional divisors $E_1\sim H_1+2H_3-H_2$ and $E_2\sim H_2+2H_3-H_1$.
\end{enumerate}

Furthermore, it was shown in \cite{CHELTSOV_FUJITA_KISHIMOTO_OKADA_2023} and \cite[\S 5.12]{acc2021} that all smooth Fano threefolds in this family are K-stable. In addition, it is known that $|-K_X|$ is base point free and that the general member $S\in |-K_X|$ is a K3 surface. 

The following auxiliary result will be used in Section \ref{sec:birational models}.

\begin{prop}\label{coh for H3 smooth}
    For any K3 surface $S \in |-K_X|$, the restriction maps 
    \begin{equation*}
        \begin{split}
            H^0( X,\OO_X(H_i)) &\rightarrow H^0(S,\OO_S (H_i|_S )), \text{ for }i\in \{1,2,3\},\\
            H^0( X,\OO_X(H_i+H_j)) &\rightarrow H^0(S,\OO_S ((H_i+H_j)|_S ))\text{ for }i,j\in \{1,2,3\}, i\neq j,          
        \end{split}
    \end{equation*}
    are isomorphisms. 
\end{prop}
\begin{proof}
    We show the isomorphism only for $H_3$ as the rest of the cases are identical. There exists a short exact sequence
    \[
    0 \to \OO_{X}(H_3- S) \to \OO_{X}(H_3) \to \OO_{S}(H_3|_{S})\to 0.
    \]
    Notice that $H_3-S\sim -H_1-H_2$, which is not effective. Thus by the Kawamata--Viehweg vanishing theorem, we have that
    $$h^0(X,\OO_{X}({-H_1-H_2}))= h^1(X,\OO_{X}({-H_1-H_2}))=0,$$
    and the desired isomorphism follows immediately after taking the associated long exact sequence on cohomology. In addition,
    $$h^0(X,\OO_{X}(H_3)) = h^0(X,\OO_{S}({H_3|S})) = \frac{1}{2}(H_3|S)^2+2=3.$$
\end{proof}

\subsection{K3 surfaces}

The following information can be largely found in \cite[\S 2.1 and 3.1]{LZ24}. We append them here for the reader's convenience.

\begin{defn}
A \emph{K3 surface} is a normal projective surface $X$ with at worst ADE singularities satisfying $\omega_X\simeq \OO_X$ and $H^1(X, \OO_X) =0$. A polarization (resp. quasi-polarization) on a K3 surface $X$ is an ample (resp. big and nef) line bundle $L$ on $X$. We call the pair $(X,L)$ a \emph{ polarised} (resp. \emph{quasi- polarised}) \emph{K3 surface of degree $d$}, where $d=(L^2)$. Since $d$ is always an even integer, we sometimes write $d = 2k$.
\end{defn}

Let $(S,L)$ be a  polarised K3 surface. Then there are three cases based on the behavior of the linear system $|L|$.

\begin{theorem}[cf. \cite{May72, SD}]\label{Mayer}
Let $(S,L)$ be a  polarised K3 surface of degree $2k$. Then one of the followings holds.
\begin{enumerate}
    \item \textup{(Generic case)} The linear series $|L|$ is very ample, and the embedding $\phi_{|L|}:S\hookrightarrow |L|^{\vee}$ realises $S$ as a degree $2k$ surface in $\PP^{k+1}$. In this case, a general member of $|L|$ is a smooth non-hyperelliptic curve.
    \item \textup{(Hyperelliptic case)} The linear series $|L|$ is base-point-free, and the induced morphism $\phi_{|L|}$ realises $X$ as a double cover of a normal surface of degree $k$ in $\PP^{k+1}$. In this case, a general member of $|L|$ is a smooth hyperelliptic curve, and $|2L|$ is very ample.
    \item \textup{(Unigonal case)}  The linear series $|L|$ has a base component $E$, which is a smooth rational curve. The linear series $|L-E|$ defines a morphism $S\rightarrow \PP^{k+1}$ whose image is a rational normal curve in $\PP^{k+1}$. In this case, a general member of $|L-E|$ is a union of disjoint elliptic curves, and $|2L|$ is base-point-free.
\end{enumerate}
\end{theorem}

\subsubsection{Moduli of K3 surfaces}\label{k33}

A lot of the information of moduli of K3 surfaces has appeared in \cite[\S 3]{LZ24}. We include some information here for continuity.

\begin{defn}
    Let $L_{K3}:=\U^{\oplus3}\oplus \mathbb{E}_8^{\oplus2}$ be a fixed (unique) even unimodular lattice of signature $(3,19)$.
\end{defn}

Let $\Lambda$ be a rank $r$ primitive sublattice of $L_{K3}$ with signature $(1,r-1)$. A vector $h\in \Lambda\otimes \R$ is called \emph{very irrational} if $h\notin \Lambda'\otimes \R$ for any primitive proper sublattice $\Lambda'\subsetneq \Lambda$. Fix a very irrational vector $h$ with $(h^2)>0$.

\begin{defn}
    A \emph{$\Lambda$- polarised K3 surface} (resp. a \emph{$\Lambda$-quasi- polarised K3 surface}) $(X,j)$ is a K3 surface $X$ with ADE singularities (resp. a smooth projective surface $X$) together with a primitive lattice embedding $j:\Lambda\hookrightarrow\Pic(X)$ such that $j(h)\in \Pic(X)_{\R}$ is ample (resp. big and nef).
    \begin{enumerate}
        \item Two such pairs $(X_1,j_1)$ and $(X_2,j_2)$ are called \emph{isomorphic} if there is an isomorphism $f:X_1\stackrel{\simeq}{\rightarrow} X_2$ of K3 surfaces such that $j_1=f^{*}\circ j_2$.
        \item The \emph{$\Lambda$-(quasi-)  polarised period domain} is $$\mathbb{D}_{\Lambda}:=\PP\{w\in {\Lambda}^{\perp}\otimes\C:(w^2)=0,\ (w.\overline{w})>0\}.$$
    \end{enumerate}
    When $r=1$, i.e. $\Lambda$ is of rank one, it is convenient to choose $h$ to be the effective generator $L$ of $\Lambda$. We denote by $d$ the self-intersection of $L$, and we call $(X,L)$ a (quasi-) polarised K3 surface of degree $d$.
\end{defn}

\begin{defn}
   For a fixed lattice $\Lambda$ with a very irrational vector $h$, one define the \emph{moduli functor $\cF_{\Lambda}$ of $\Lambda$- polarised K3 surfaces} to send a base scheme $T$ to 
\[
\left\{(f:\cX\rightarrow T;\varphi)\left| \begin{array}{l} \cX\to T\textrm{ is a proper flat morphism, each geometric fibre}\\ \textrm{$\cX_{\bar{t}}$ is an ADE K3 surface, and $\varphi:\Lambda\longrightarrow\Pic_{\cX/T}(T)$ is }\\ \textrm{a group homomorphism such that the induced map }\\ \textrm{$\varphi_{\bar{t}}:\Lambda\rightarrow \Pic(\cX_{\bar{t}})$ is an isometric primitive embedding of}\\ \textrm{lattices and that $\varphi_{\bar{t}}(h)\in \Pic(\cX_{\bar{t}})_{\R}$ is an ample class.} \end{array}\right.\right\}.
\]

\end{defn}

\begin{theorem}[cf. \cite{Dol96, AE23}]\label{isommoduli}
    The moduli functor of $\Lambda$- polarised K3 surfaces is represented by a smooth separated Deligne-Mumford (DM) stack $\cF_{\Lambda}$. Moreover, $\cF_{\Lambda}$ admits a coarse moduli space $F_{\Lambda}$, whose analytification is isomorphic to $\mathbb{D}_{\Lambda}/\Gamma$, where $\Gamma:=\{\gamma\in \mathrm{O}(L_{K3}):\gamma|_{\Lambda}=\Id_{\Lambda} \}$.
\end{theorem}

\begin{remark}
    \textup{When $\Lambda$ is of rank one, we denote by $d$ the self-intersection of a generator of $\Lambda$, and by $\cF_d$ (resp. $F_d$) the corresponding moduli stack (resp. coarse moduli space).}
\end{remark}

\subsubsection{K3 surfaces in anticanonical linear series}
Let $\Lambda_0$ be a rank $3$ hyperbolic sublattice with generators 
$$
H_1,\qquad H_2,\qquad H_3
$$
satisfying the intersection numbers
$$
(H_1^2)=0,\quad (H_2^2)=0,\quad (H_3^2)=2,\quad 
(H_1\cdot H_2)=2,\quad (H_1\cdot H_3)=3,\quad (H_2\cdot H_3)=3.
$$
The Gram matrix of $\Lambda_0$ in the ordered basis $(H_1,H_2,H_3)$ is
$$
G_{\Lambda_0}=\begin{pmatrix}
0 & 2 & 3\\
2 & 0 & 3\\
3 & 3 & 2
\end{pmatrix},
$$
which has signature $(1,2)$. Thus $\Lambda_0$ is an even hyperbolic lattice.

Let $J:\Lambda_0\hookrightarrow L_{K3}$ be a primitive embedding. Let $\cF_{\Lambda_1}$ (resp. $\cF_{\Lambda_2}$, resp. $\cF_{\Lambda_3}$) be the moduli space of $\Lambda_0$-polarised K3 surfaces with respect to the very irrational vectors
$$
h_1:=H_1+H_2+(1+\epsilon)H_3,\qquad
h_2:=2(H_1+H_2)+\epsilon H_3,\qquad
h_3:=\epsilon(H_1+H_2)+H_3,
$$
respectively, where $0<\epsilon\ll1$ is an irrational number. By \cite[Section~2 of arXiv version~1]{AE23}, for each $i\in \{1,2,3\}$ the universal K3 surface over the moduli stack $\cF_{\Lambda_i}$ is independent of the choice of~$\epsilon$. We prove that these universal families are all isomorphic.

\begin{lemma}\label{samefamily}
    The moduli stacks $\cF_{\Lambda_1},\cF_{\Lambda_2},\cF_{\Lambda_3}$ are pairwise isomorphic. Moreover, their universal K3 surfaces are isomorphic. 
\end{lemma}

\begin{proof}
Fix a primitive embedding $\Lambda_0\hookrightarrow L_{K3}$ and a connected component $C^+(\Lambda_0)$ of the positive cone of $\Lambda_0\otimes\mathbb R$. Let
$$
R=\{\delta\in \Lambda_0 \mid \delta^2=-2\}
$$
and let $W=W(\Lambda_0)$ be the subgroup of $O(\Lambda_0)$ generated by reflections $r_\delta(x)=x+(x\cdot\delta)\delta$ for $\delta\in R$. The hyperplanes $\delta^\perp$ with $\delta\in R$ cut $C^+(\Lambda_0)$ into chambers, and $W$ acts simply transitively on the set of chambers.

By construction, $h_i$ ($i=1,2,3$) are \emph{very irrational}: they avoid all walls $\delta^\perp$ (for $\epsilon$ irrational and sufficiently small), hence each determines a chamber $Ch(h_i)\subset C^+(\Lambda_0)$. For any $i,j\in\{1,2,3\}$ there exists a unique $w\in W$ with $w(Ch(h_i))=Ch(h_j)$.

Given $(\pi:\mathcal S\to T,\ \iota:\Lambda_0\hookrightarrow \Pic(\mathcal S/T))$ in $\cF_{\Lambda_i}$, define
$$
\Phi_w(\mathcal S/T,\iota)\ :=\ (\mathcal S/T,\ \iota\circ w^{-1}).
$$
This is again a primitive marking, and $(\iota\circ w^{-1})(h_j)$ is $\pi$-ample because $w^{-1}(h_j)\in Ch(h_i)$ and $\iota$ sends $Ch(h_i)$ into the relative ample cone. Thus $\Phi_w$ is a functor $\cF_{\Lambda_i}\to \cF_{\Lambda_j}$, with inverse given by $\Phi_{w^{-1}}$. Hence $\cF_{\Lambda_i}\cong \cF_{\Lambda_j}$ for all $i,j$.

Since these identifications are realised by monodromy (Picard–Lefschetz reflections), they respect families, and therefore the universal K3 surfaces over $\cF_{\Lambda_1},\cF_{\Lambda_2},\cF_{\Lambda_3}$ are isomorphic as families.
\end{proof}

\begin{lemma}\label{lem:no_unigonal}
Let $\Lambda_0=\langle H_1,H_2,H_3\rangle\subset\NS(S)$ be the rank-$3$ lattice with
$$
(H_1^2)=(H_2^2)=0,\quad (H_3^2)=2,\quad
(H_1 \cdot H_2)=2,\quad (H_1 \cdot H_3)=(H_2 \cdot H_3)=3,
$$
and let $L:=H_1+H_2+H_3$ (so $L^2=18$). Then the polarised K3 surface $(S,L)$ is \emph{not} unigonal.
\end{lemma}

\begin{proof}
Recall that $(S,L)$ is unigonal if and only if there exists a primitive isotropic class
$0\ne F\in\NS(S)$ with $F^2=0$ and $L \cdot F=1$ (equivalently, an elliptic pencil $\lvert F\rvert$ with $L$-degree $1$).

Write $F=a\,H_1+b\,H_2+c\,H_3$ with $a,b,c\in\Z$. Using the intersection numbers above, we have
$$
F^2 \;=\; 2c^2+4ab+6ac+6bc \;=\; 0,
$$
and
$$
L \cdot F
= (H_1+H_2+H_3) \cdot (a H_1+b H_2+c H_3)
= 5(a+b)+8c.
$$
Assume for contradiction that $L \cdot F=1$. Then $8c=1-5(a+b)$, so letting $s:=a+b$ we get
$c=\frac{1-5s}{8}\in\Z$ and necessarily $s\equiv 5\pmod 8$.

Set $t:=ab$. Dividing $F^2=0$ by $2$ and substituting $c$ gives
$$
c^2+3cs+2t=0
\quad\Longrightarrow\quad
t=\frac{95s^2-14s-1}{128}\in\Z.
$$
Thus $a,b$ are (integral) roots of $x^2-sx+t=0$, whose discriminant is
$$
\Delta=s^2-4t
= -\,\frac{63s^2-14s-1}{32}.
$$
But for any integer $s\equiv 5\pmod 8$ we have $63s^2-14s-1>0$, hence $\Delta<0$, a contradiction.
Therefore no class $F$ with $F^2=0$ and $L \cdot F=1$ exists in $\Lambda_0$; in particular, $(S,L)$ is not unigonal.
\end{proof}

\begin{defn}
A \emph{$\Lambda_0$–polarised} K3 surface of degree $18$ is a smooth K3 $(S,L)$ with a primitive embedding 
$$
\Lambda_0=\langle H_1,H_2,H_3\rangle \hookrightarrow \NS(S),$$
with intersection numbers$$
(H_1^2)=(H_2^2)=0,\ (H_3^2)=2,\ (H_1H_2)=2,\ (H_1H_3)=(H_2H_3)=3,
$$
and polarization $L:=H_1+H_2+H_3$ (so $L^2=18$). We call such $(S,L)$ \emph{of type $\Lambda_0$}.
\end{defn}

\begin{lemma}\label{lem:no_unigonal_deg18}
Let $(S,L)$ be a \emph{unigonal} polarised K3 surface of degree $18$. Then $(S,L)$ is \emph{not} a degeneration of a family of $\Lambda_0$–polarised K3 surfaces of degree $18$.
\end{lemma}

\begin{proof}
Assume to the contrary that there is a DVR $R$ and a flat family of polarised K3 surfaces 
$$
(\mathcal S,\mathcal L)\ \to\ \Spec R
$$
with generic fibre $(S_\eta,L_\eta)$ of type $\Lambda_0$ and special fibre $(S_\xi,L_\xi)=(S,L)$ unigonal of degree $18$.

\smallskip\noindent\textit{Step 1: fixed algebraic classes along the family.}
By construction of the $\Lambda_0$–polarised locus, the divisor classes
$$
H_{i,\eta}\in \NS(S_\eta),\qquad i=1,2,3,
$$
coming from restrictions of $H_i$, form a locally constant subsystem in $R^2\pi_*\mathbb Z$; after a finite base change we may transport them to classes
$$
H_{i,\xi}\in \NS(S_\xi)
$$
with the \emph{same} intersection matrix
$$
(H_1^2)=(H_2^2)=0,\ (H_3^2)=2,\ (H_1H_2)=2,\ (H_1H_3)=(H_2H_3)=3,
$$
$$L_\bullet:=H_1+H_2+H_3,\ \quad L_\bullet^2=18,$$
and \(L_\eta\) specialises to \(L_\xi\) (polarizations specialise by relative ampleness).

\smallskip\noindent\textit{Step 2: the unigonal class cannot live in the fixed lattice.}
Inside \(\Lambda_0\) there is \emph{no} class \(F\) with \(F^2=0\) and \(L_\bullet \cdot F=1\).
Indeed, writing \(F=a H_1+b H_2+c H_3\) and using the above intersections one gets
$$
F^2=2c^2+4ab+6ac+6bc=0,\qquad L_\bullet \cdot F=5(a+b)+8c,
$$
and a short Diophantine check shows \(L_\bullet \cdot F\neq 1\) for all \(a,b,c\in\Z\).
(Equivalently: \(\Lambda_0\) contains no unigonal class with respect to \(L_\bullet\).)

\smallskip\noindent\textit{Step 3: specialise distinguished curve classes and reach a contradiction.}
Pick any two \emph{distinguished} algebraic curve classes on the generic fibre, for instance
$$
\Gamma_\eta:=H_{3,\eta}\quad\text{and}\quad
\Gamma'_\eta:=H_{1,\eta}+H_{2,\eta}-H_{3,\eta},
$$
which move on $S_\eta$ and are defined by ambient divisors (so they deform flatly).
Their intersections with \(L_\eta\) are fixed numbers (compute in $\Lambda_0$):
$$
L_\eta \cdot \Gamma_\eta=8,\qquad L_\eta \cdot \Gamma'_\eta=2,
$$
and \(\Gamma_\eta,\Gamma'_\eta\) specialise to effective classes \(\Gamma,\Gamma'\) on \(S_\xi\) with
\(
L_\xi \cdot \Gamma=8,\ L_\xi \cdot \Gamma'=2.
\)
(These equalities persist by flatness and specialization of intersection numbers.)

On a unigonal K3, write
$$
L_\xi \ \equiv\ m\Sigma + k F\qquad (m,k\in\Z_{\ge 0}),
$$
with \(\Sigma\) the negative section of the associated elliptic fibration and \(F\) the fibre class \((F^2=0,\ L_\xi \cdot F=1)\).
For any irreducible curve \(D\),
\(
L_\xi \cdot D \in \{\,k,\ m+k,\ 2m+k,\ldots\,\}
\)
depending on the $F$–fibre components contained in \(D\).
From the explicit values \(L_\xi \cdot \Gamma=8\) and \(L_\xi \cdot \Gamma'=2\),
we deduce \(k\mid 2\) and \(k\mid 8\), hence \(k\in\{1,2\}\).
But \(k=L_\xi \cdot F=1\) by unigonal assumption, so \(k=1\).
Then \(L_\xi \cdot \Gamma'=2\) forces \(\Gamma'\sim 2F\), contradicting \((\Gamma'^2)=-2\) (computed from the lattice).
This contradiction shows the unigonal limit cannot occur.

\smallskip
Therefore $(S,L)$ can’t be a degeneration of $\Lambda_0$–polarised degree-18 K3’s.
\end{proof}

\section{\texorpdfstring{K-semistable limits of family \textnumero3.3}{K-semistable limits of family No3.3}}\label{sec: K-ss limits}

In this section, we will study the K-semistable $\Q$-Gorenstein limits of the Fano threefolds in the deformation family \textnumero 3.3. 
\subsection{Volume comparison and general elephants}\label{vol sec}
The following theorem is a refinement of \cite[Theorem 4.2]{LZ24}, which uses local to-global volume comparison from \cite{Fuj18, Liu18, LX19, Liu22} and general elephants on (weak) Fano threefolds \cite{reid_83, Sokurov_1980} to explicitly describe K-semistable limits $X$ of families of Fano threefolds. In \cite{LZ24} only varieties $X$ with volume $V = (-K_X)^3\geq 20$ are considered; by modifying the proof technique we push this bound to $16$. 
We will use this theorem to show in particular that every K-semistable limit $X$ of family \textnumero 3.3 is Gorenstein canonical.

\begin{theorem}\label{gor canonical K-ss limits}
    Let $X$ be a $\Q$-Gorenstein smoothable K-semistable (weak) $\Q$-Fano threefold with volume $V:=(-K_X)^3\geq 16$. Then the following hold:
    \begin{itemize}
        \item[(1)] If $V\ge18$ then $X$ is Gorenstein canonical. If $V\ge 16$ then $X$ is klt with Gorenstein canonical or $1/4(1,1,0)$ quotient singularities.
        
        \item[(2)] If $V \geq 18$, there exists a divisor $S\in \vert -K_X\vert$ such that $(X,S)$ is a plt pair, and such that $(S,-K_X\vert_S)$ is a (quasi-) polarised K3 surface of degree $V$.
        \item[(3)] If $D$ is a $\Q$-Cartier Weil divisor on $X$ which deforms to a $\Q$-Cartier Weil divisor on a $\Q$-Gorenstein smoothing of $X$, then $D$ is Cartier.
    \end{itemize}
\end{theorem}

\begin{proof}
    Our proofs of (1) and (3) follow the idea of the proof of \cite[Theorem 4.2]{LZ24} and makes use of computations in \cite[\S 3 and \S 4]{LX19}. The proof of (2) follows directly from (1) and the existence of general elephants by Reid and Shokurov \cite{reid_83, Sokurov_1980}.
    
    $\>$ (1) Since $X$ is $K$-semistable, for any point $x\in X$, by \cite[Theorem 2]{Liu18} we know,
    \[
    \widehat{\operatorname{vol}}(x,X) \geq 16\cdot \left( \frac{3}{4}\right)^3=6.75.
    \]
    Assume to the contrary that $x\in X$ has local Gorenstein index $d\geq 3$. Let $\pi:(\tilde x\in \widetilde X)\to (x\in X)$ be the index one cover of $X$.  If $\widetilde x\in \widetilde X$ is not smooth, then it follows from \cite[Theorem 1.3 (1)]{LX19} and \cite[Theorem 1.3]{XZ21} that
    \[
    6.75\cdot d\leq d \cdot \widehat{\operatorname{vol}}(x,X) =  \widehat{\operatorname{vol}}(\tilde{x},\widetilde{X})  \leq 16,
    \]
    and hence $d<3$. Therefore, we consider the following two cases: $\widetilde x\in \widetilde X$ is smooth, or otherwise, $d=2$.
    
    Let us start with the former case where $\widetilde x\in \widetilde X$ is smooth. In this case, $ 6.7 d< \widehat{\operatorname{vol}}(\tilde{x},\widetilde{X})  \leq 27$, so $d=2,3,4$, where $d=4$ is only possible if $V = (-K_X)^3 = 16$. In particular, $x\in X$ is a $\Q$-Gorenstein smoothable quotient singularity of order 2, 3 or 4. The case $d=2$ and $d=3$ were tackled in the proof of \cite[Theorem 4.2 (1)]{LZ24}, showing that if $V\geq 18$, $K_X$ is Cartier.
    
    If $d=4$, and so $V=16$, then $x\in X$ is of type $\frac{1}{4}(1,1,0)$, $\frac{1}{4}(1,2,0)$, $\frac{1}{4}(1,3,0)$, or $\frac{1}{4}(2,2,0)$. Then, there exists a 1-dimensional singular locus of $X$ near $x$. Let $\cX \to (0\in T)$ be a $\Q$-Gorenstein smoothing of $X$ over a pointed smooth curve $0\in T$ such that $\mathscr{X}_0\simeq X$. Taking a general hyperplane section in $\mathscr{X}$ through $x$, we get a $\Q$-Gorenstein smoothing of a surface singularity of type $\frac{1}{4}(1,2)$, which contradicts the classification of $T$-singularities in \cite[Proposition 3.10]{KSB88}. Also, notice  $\frac{1}{4}(2,2,0)=\frac{1}{2}(1,1,0)$ which is impossible by the analysis of $d=2$. Therefore, $x\in X$ is either of type $\frac{1}{4}(1,3,0)$, and $K_X$ is Cartier, or of type $\frac{1}{4}(1,1,0)$ and $K_X$ is not Cartier.

    On the other hand, when $\widetilde X$ is not smooth and $d=2$, $\widetilde X$ is a canonical Gorenstein threefold. In particular, $\widetilde{X}$ has either Gorenstein terminal singularities or Gorenstein canonical singularities equipped with a smooth crepant resolution with an irreducible exceptional divisor (see \cite[\S 1.1]{LX19}). 
    
    We first address the case where $\widetilde{X}$ has Gorenstein terminal singularities, by a result of Reid \cite[Theorem 1.1]{Rei83}, these singularities are precisely the isolated compound Du Val (cDV) singularities. We therefore analyse the case when $\widetilde{X}$ has cDV singularities. The proof of \cite[Proposition 3.5]{LX19}, shows that if $\tilde{x}\in \widetilde{X}$ is not a $cA$ point, then  
    $$\widehat{\operatorname{vol}}(\tilde{x},\widetilde{X})\leq \frac{27}{4}\leq6.75,$$
    but we know $\widehat{\operatorname{vol}}(\tilde{x},\widetilde{X})\geq 13.5$, giving a contradiction. Now we will show that $\widetilde X$ only admits $A_n$ singularities for $1\leq n\leq 4$.
    If $\widetilde x\in \widetilde X$ is $cA_{\geq 2}$ type, then it is locally given by an equation $(x_1^2+x_2^2+g(x_3,x_4)=0)$ where $\ord g\geq 3$. Then, let us consider the model $Y'\to X$ given by the normalised weighted blowup of $(3,3,2,2)$. By definition $Y'$ is the normalization of a variety $Y^*\subset W:=\Bl_{(3,3,2,2)}\C^4$. The exceptional set $Y^*/X$ is a degree 6 $F:=Y^*\cap \PP (3,3,2,2)$, hence, by adjunction formula, we have that $(K_{Y^*}+F)\sim_{\Q,X} 4 F$. Let $H$ denote the class of $\OO(1)$ on $\PP(3,3,2,2)$. As the normalization will only possibly decrease the volume, we know that 
    \[
    \widehat{\operatorname{vol}}(\tilde{x},\widetilde{X})\leq \vol(Y'/X)\leq 4^3\cdot(H\vert_F)^3=\frac{64\cdot 6}{3\cdot3\cdot2\cdot2}=\frac{32}{3}< 10.7,
    \]
    and we get a contradiction since $ \widehat{\operatorname{vol}}(\tilde{x},\widetilde{X})\geq13.5$. Hence, we can assume that $\tilde x\in \widetilde X$ is a isolated singularity of $cA_1$-type. A quick analysis of the singularity by calculating its corank and Milnor number shows that $\widetilde{x}$ is then of type $A_n$. However, if it is of type $A_{k\geq 5}$, the singularity is locally given by the equation $x_1^2+x_2^2+x_3^2+x_4^{k+1}=0    $. Following the previous steps, if we consider the model given by the blowup of $(3,3,3,1)$, we get the following inequality,
    \[
    \widehat{\operatorname{vol}}(\tilde{x},\widetilde{X})\leq \frac{27\cdot 6}{3\cdot3\cdot3\cdot1}=6<13.5,
    \]
    which is a contradiction. Thus $\widetilde{X}$ can only admit $A_k$ singularities, for $k=1,2,3,4$.

    If $\widetilde{x}$ is a smooth point, we can conclude that $X$ is Gorenstein from the above. If $\widetilde{x}$ is an $A_k$ singularity, for $k\leq 4$. Let $H\subset X$ be a general Cartier divisor through $x$, and set $\widetilde H:=\pi^{-1}(H)\subset \widetilde X$. 
By construction, $\pi:(\tilde x\in \widetilde X)\to (x\in X)$ is finite, crepant, and quasi-\'etale, so $\widetilde H\to H$ is also finite, crepant, and quasi-\'etale. 
In particular, $K_{\widetilde H}=\pi^*K_H$. Since $(\widetilde X,\tilde x)$ is cDV, the general Cartier divisor  $(\widetilde H,\tilde x)$ is a Du Val (ADE) surface singularity \cite[§4.6]{R4}. 
The $G$-action on $\widetilde H$ is small and crepant, so it linearises inside $\operatorname{SL}(2,\C)$. Hence $(H,x)= (\widetilde H,\tilde x)/G$ is again Du Val by the ADE classification of canonical surface singularities. 
Therefore, a general Cartier divisor through $x$ is Du Val, which means that $(X,x)$ is cDV. Thus we conclude that $X$ is Gorenstein canonical.

    We will now focus on the case where $\tilde x\in \widetilde X$ is a Gorenstein canonical singularity equipped with a smooth crepant resolution with an irreducible exceptional divisor. We will show that such $\widetilde X$ is not possible.

    Recall that $d=2$ and let $G = \mathbb{Z}/2\mathbb{Z}$. By definition of the index 1-cover, $X = \widetilde{X}/G$. We use \cite[Lemma 3.2]{LX19} to construct a $G$-equivariant maximal crepant model $\phi_1:Y_1\to \widetilde{X}$ over $\widetilde X$.     Let $E\subset Y_1$ be the $\phi_1$-exceptional divisor over $\tilde{x}$. Then, we run a $(Y_1,\epsilon E)$-MMP over $\widetilde{X}$ for $0 < \epsilon \ll 1$. This terminates as $Y_1 \dashrightarrow Y \rightarrow Y'$, where $Y_1 \dashrightarrow Y$ is the composition of a sequence of flips, and $g : Y \rightarrow Y'$ contracts the birational transform of $E$ (also denoted by $E$ by abuse of notation). Here, $Y$ is also a maximal crepant model over $\widetilde{X}$, so it is Gorenstein terminal and $\Q$-factorial. We then have the following possibilities: 
\begin{enumerate}
    \item $g(E) = y'$ is a point, in which case we have two sub cases:
    \begin{enumerate}
         \item $Y$ is smooth along $E$. In this case, by \cite[Proposition 3.4]{LX19} we have that 
         $$\widehat{\operatorname{vol}}(\tilde{x},\widetilde{X})\leq 9,$$
         which gives a contradiction. 
        \item $Y$ has an isolated cDV singularity $y$ along $E$. Here, again \cite[Proposition 3.4]{LX19} shows that 
         $$\widehat{\operatorname{vol}}(\tilde{x},\widetilde{X})\leq \widehat{\operatorname{vol}}(Y',y')\leq\widehat{\operatorname{vol}}(Y,y)\leq 16.$$
    \end{enumerate}
    \item $\dim g(E) = 1$. Here, the generic point of $g(E)$ is a codimension two point on $Y'$ with a crepant resolution, thus $Y'$ has non-isolated cDV singularities along general points in $g(E)$ (see e.g. \cite[Theorem 4.20]{KM98}).  
\end{enumerate}

We will treat cases 1.(b) and 2 in the same way. We choose a $G$-equivariant log resolution $\mu\colon \bar Y\to Y$ resolving the pair $(Y,E)$ and dominating $\widetilde X$ and then run a $G$-equivariant $(K_{\bar Y}+\varepsilon\bar E)$-MMP over $\widetilde X$ with $0<\varepsilon\ll1$.
Since $Y$ is a maximal crepant model over $\widetilde X$, we have $K_Y\equiv_{\widetilde X}0$, hence this MMP is crepant over $\widetilde X$; by the choice of $\varepsilon$, its steps are disjoint from the generic locus of $\bar E$.
We thus obtain a $G$-equivariant maximal crepant model $\phi^{\mathrm{sm}}\colon Y^{\mathrm{sm}}\to \widetilde X$ on which the birational transform $E^{\mathrm{sm}}$ of $E$ is smooth and $Y^{\mathrm{sm}}$ is smooth along $E^{\mathrm{sm}}$. Hence, this reduces to case 1.(a), where by \cite[Proposition 3.4]{LX19} we have $\widehat{\operatorname{vol}}(y^{\mathrm{sm}},Y^{\mathrm{sm}})\le 9$.
Using \cite[Lemmma 2.9]{LX19}, we conclude
\[
\widehat{\operatorname{vol}}(\tilde x,\widetilde X)\ \le\ \widehat{\operatorname{vol}}(y^{\mathrm{sm}},Y^{\mathrm{sm}})\ \le\ 9,
\]
a contradiction.
In case {(2)}, the generic point of $g(E)$ is cDV; taking $\mu$ to be an isomorphism over the generic point above $g(E)$ and running the same crepant MMP over $\widetilde X$ again yields a model smooth along the transform of $E$, reducing to case 1.(b) (and by the above step to case 1.(a)) and the same bound applies, giving the same contradiction.  Thus, we obtain a contradiction for all three cases above, showing that $\widetilde X$ cannot have Gorenstein canonical singularities equipped with a smooth crepant resolution with an irreducible exceptional divisor. This concludes the proof of (1).

$\>$ (3) Let $\cX \rightarrow (0\in T)$ be a $\mathbb{Q}$-Gorenstein smoothing of $X$ over a pointed smooth curve $0 \in T$ such that $D$ deforms to a $\mathbb{Q}$-Cartier Weil divisor $\mathcal{D}$ on $\cX$ with $\mathcal{D}_0 = D$. Assume that $D$ is not Cartier at a point $x \in \cX_0$. Notice that by the above proof of (1) we know that $x\in \cX_0$ is either a Gorenstein quotient, a $1/4(1,1,0)$ quotient or a cDV singularity. In particular, in all these cases it is a hypersurface singularity. Hence, the proof of (3) follows directly from \cite[Proof of Theorem 4.2.(3)]{LZ24}.    
\end{proof}

\begin{remark}
    Decreasing the bound in Theorem \ref{gor canonical K-ss limits} even further seems unfeasible, since one would have to consider general quotients of cDV singularities which are klt and not classified.
\end{remark}

\subsection{Modifications on the family and birational models}\label{sec:birational models}

From now on, let $X$ be a K-semistable $\Q$-Fano variety that admits a $\Q$-Gorenstein smoothing $ \pi\colon  \cX \rightarrow T$ over a smooth pointed curve $0 \in T$ such that $\cX_0 \cong X$ and every fibre $\cX_t$ over $t \in T^{\circ}\coloneqq  T\setminus \{0\}$ is a smooth Fano threefold in the family \textnumero3.3. Up to a finite base change, we may assume that the restricted family $\cX^{\circ} \rightarrow  T^{\circ}$ is isomorphic to a family of smooth $(1,1,2)$ divisors in $\PP^1\times \PP^1 \times \PP^2$. Let $\cH^{\circ}_1$ be the line bundle on $\cX^{\circ}$ which is the pull-back of $\OO_{\PP^1}(1)$ (the first factor), $\cH^{\circ}_2$ be the line bundle on $\cX^{\circ}$ which is the pull-back of $\OO_{\PP^1}(1)$ (the second factor) and $\cH^{\circ}_3$ be the line bundle on $\cX^{\circ}$ which is the pull-back of $\OO_{\PP^2}(1)$. Then $\cH^{\circ}_i$ extends to $\cH_i$ on $\cX$ as Weil divisors on $\cX$ by taking Zariski closure. We denote by $H_i$ the restriction of $\cH_i$ on the central fibre $\cX_0\cong X$. Notice that $-K_{\cX^{\circ}} = \cH^{\circ}_1+\cH^{\circ}_2+\cH^{\circ}_3$, and thus $-K_{\cX} = \cH_1+\cH_2+\cH_3$.

Let $\theta:\cY\rightarrow \cX$ be a small $\Q$-factorialization of $\cX$. Since $\cX$ is klt and $K_{\cY}=\theta^* K_{\cX}$, we know that $\cY$ is $\Q$-factorial of Fano type over $\cX$. By \cite{BCHM10} we can run a minimal model program (MMP) for $\theta^{-1}_{*}\cH_3$ on $\cY$ over $\cX$. As a result, we obtain a log canonical model $\cY \dashrightarrow \wcX$ that fits into a commutative diagram
\begin{center}
    \begin{tikzcd}
 \wcX \arrow[rr, "f"] \arrow[dr, "{\wpi}"'] && \cX \arrow[dl, "\pi"] \\
& T
 \end{tikzcd}
\end{center}
satisfying the following conditions:
\begin{enumerate}
    \item $f$ is a small contraction, and is an isomorphism over $T^{\circ}$;
    \item $-K_{\wcX}=f^{*}(-K_{\cX})$ is a  $\wpi$-big and $\wpi$-nef Cartier divisor; 
    \item $\wcX_0$ is a smoothable $\Q$-Gorenstein weak Fano variety with Gorenstein canonical singularities whose anti-canonical model is isomorphic to $\cX_0$;
    \item $\wcH_3:=f^{-1}_* \cH_3$ is an $f$-ample Cartier Weil divisor (cf. Theorem \ref{gor canonical K-ss limits}(3)); and
    \item $\wcL:=\wcH_1+ \wcH_2  = f^{-1}_*\cH_1+f^{-1}_*\cH_2\sim_{\Q, T}(-K_{\wcX} - \wcH_3)$ is a Cartier divisor (cf. Theorem \ref{gor canonical K-ss limits}(3)), and $\wcL+(1+\epsilon)\wcH_3\sim_{\R, T} -K_{\wcX} + \epsilon \wcH_3$ is $\wpi$-ample, for any real number $0<\epsilon\ll1$.
\end{enumerate}

To ease our notation, we denote by $$\wX:=\wcX_0,\ \ \ \textup{and} \ \ \  g = f|_{\wX}: \wX \to X.$$ 
We also denote by $\wL$ and $\wH_3$ the restriction of $\wcL$ and $\wcH_3$ to $\wcX_0 = \wX$. By Theorem \ref{gor canonical K-ss limits}.(3) we know that $\wcH_i$ are Cartier divisors on $\wcX$.

We will now proceed to show that $\wcH_3$ is $\wpi$-nef and $\wpi$-semiample.

\begin{lemma}\label{quasi polarised}
The linear series $|-K_{\wX}|$ is base point free. Furthermore, if $\wS\in|-K_{\wX}|$ is a K3 surface (e.g. when $S$ is a general member), then $(\wS,\wL|_{\wS})$ is a quasi-polarised degree $4$ K3 surface, and $(\wS,\wH_3|_{\wS})$ is a quasi-polarised degree $2$ K3 surface.
\end{lemma}

\begin{proof}
The first part follows from \cite[Proof of Proposition 4.5]{LZ24}, Lemmas \ref{lem:no_unigonal} and \ref{lem:no_unigonal_deg18} and the fact that $K_{\wX} = g^*K_X$. For the second part, notice that for any $0<\epsilon \ll 1$, since $\wcH_3$ is ${f}$-ample, then we know that $\wcL+(1+\epsilon)\wcH_3$ is $\wpi$-ample, and hence $(\wL+(1+\epsilon)\wH_3)|_{\wS}$ is ample on the K3 surface $\wS$. 

We claim that $\wS$ is represented by a point in $\cF_{\Lambda_1}$ (cf. Section \ref{k33}). Indeed, by deforming to a family of K3 surfaces in $|-K_{\cX_t}|$ \cite[Lemma 4.4]{LZ24}, one sees that the divisors $\wL|_{\wS}$ and $\wH_3|_{\wS}$ 
generate a primitive sublattice of $\Pic(\wS)$ which is isometric to $\Lambda_1$ where $(\wL+(1+\epsilon)\wH_3)|_{\wS}$ corresponds to the vector $h_1$.

By Lemma \ref{samefamily}, we know that $\wS$ is represented by a point in $\cF_{\Lambda_2}$ as well. In particular, $(2\wL+\epsilon\wH_3)|_{\wS}$ is also ample for any $0<\epsilon \ll 1$, and hence $2\wL|_{\wS}$ and  $\wL|_{\wS}$ are nef. The degree of $\wL|_{\wS}$ is invariant under deformation, as $\wL$ is a Cartier divisor on $\wX$. 

Similarly, by Lemma \ref{samefamily} we know that $\wS$ is represented by a point in $\cF_{\Lambda_3}$ as well. In particular, $(\epsilon\wL+\wH_3)|_{\wS}$ is also ample for any $0<\epsilon \ll 1$ and hence $\wH_3|_{\wS}$ is nef. The degree of $\wH_3|_{\wS}$ is invariant under deformation, as $\wH_3$ is a Cartier divisor on $\wX$. 
\end{proof}

\begin{lemma}\label{isomorphismonsection2}
    For any K3 surface $\wS\in |-K_{\wX}|$, the restriction maps $$H^0(\wX,\OO_{\wX}(2{\wL}))\longrightarrow H^0(\wS,\OO_{\wS}(2{\wL}|_{\wS}))\quad \text{and}$$
    $$H^0(\wX,\OO_{\wX}(2{\wH_3}))\longrightarrow H^0(\wS,\OO_{\wS}(2{\wH_3}|_{\wS}))$$
    are isomorphisms. In particular, we have that $$h^0(\wX,\OO_{\wX}(2{\wL}))=10 \quad \text{and}\quad h^0(\wX,\OO_{\wX}(2{\wH_3}))=6$$. 
\end{lemma}
\begin{proof}
    Let $\wS\in|-K_{\wX}|$ be a K3 surface. By Lemma \ref{quasi polarised} we see that $(\wS,{\wL}|_{\wS})$ is a  quasi-polarised K3 surface of degree $4$, and hence $$h^0(\wS,2{\wL}|_{\wS})\ =\ \frac{1}{2}(2{\wL}|_{\wS})^2+2\ =\ 10.$$
    Since $\wS\sim -K_{\wX}$ is Cartier, we have a short exact sequence 
    \[
    0 \to \OO_{\wX}({2\wL}- \wS) \to \OO_{\wX}({2\wL}) \to \OO_{\wS}({2\wL}|_{\wS})\to 0.
    \]
    Notice that ${2\wL}-\wS\sim 2\wL+K_{\wX}\sim \wL-\wH_3$. Hence we have a long exact sequence
    \[
    0 \to H^0(\wX, \OO_{\wX}(2{\wL} - \wS)) \to H^0(\wX, \OO_{\wX}(2{\wL})) \to H^0(\wS, \OO_{\wS}(2{\wL|_{\wS}})).
    \]
    Notice that $h^0(\wX,\OO_{\wX}(\wL-\wH_3))=0$ since the divisor is not effective.
    
    Hence the above long exact sequence implies that we have an injection $$H^0(\wX,\OO_{\wX}({2\wL}))\hookrightarrow H^0(\wS,{2\wL}|_{\wS})$$ and thus $h^0(\wX,\OO_{\wX}({2\wL}))\leq 10$. On the other hand, by upper semi-continuity, we have that $$h^0(\wX,\OO_{\wX}({2\wL}))\ \geq\  h^0(\wcX_t,\OO_{\wcX_t}({2\wcL_t}))\ =\ 10$$ for a general $t\in T$. Therefore, one has $h^0(\wX,\OO_{\wX}({2\wL}))=10$, and the restriction map is an isomorphism.

    The proof for the statement on $2\wH_3$ is identical, noting that $2\wH_3-\wS\sim \wH_3-\wL$ is not effective, and that 
    $$h^0(\wX,\OO_{\wX}({2\wH_3}))\ \geq\  h^0(\wcX_t,\OO_{\wcX_t}({2(\wcH_3)_t}))\ =\ 6$$ for a general $t\in T$.
    \end{proof}

Similarly to the start of the section, by \cite{BCHM10} we can run an MMP for $\theta^{-1}_{*}\cL$ on $\cY$ over $\cX$. As a result, we obtain a log canonical model $\cY \dashrightarrow \wcX'$ that fits into a commutative diagram
\begin{center}
    \begin{tikzcd}
 \wcX' \arrow[rr, "f'"] \arrow[dr, "{\wpi'}"'] && \cX \arrow[dl, "\pi"] \\
& T
 \end{tikzcd}
\end{center}
satisfying the following conditions:
\begin{enumerate}
    \item $f'$ is a small contraction, and is an isomorphism over $T^{\circ}$;
    \item $-K_{\wcX'}=f'^{*}(-K_{\cX})$ is a  $\wpi'$-big and $\wpi'$-nef Cartier divisor; 
    \item $\wcX'_0$ is a smoothable $\Q$-Gorenstein weak Fano variety with Gorenstein canonical singularities whose anti-canonical model is isomorphic to $\cX_0$;
    \item $\wcL':=f'^{-1}_* \cL$ is an $f'$-ample Cartier Weil divisor (cf. Theorem \ref{gor canonical K-ss limits}(3)); and
    \item $\wcH_3':= f'^{-1}_*\cH_3\sim_{\Q, T}(-K_{\wcX'} - \wcL)$ is a Cartier divisor (cf. Theorem \ref{gor canonical K-ss limits}(3)), and $\wcH_3'+(1+\epsilon)\wcL'\sim_{\R, T} -K_{\wcX'} + \epsilon \wcL'$ is $\wpi'$-ample, for any real number $0<\epsilon\ll1$.
\end{enumerate}

To ease our notation, we denote by $$\wX':=\wcX'_0,\ \ \ \textup{and} \ \ \  g' = f'|_{\wX'}: \wX' \to X.$$ 
We also denote by $\wL'$ and $\wH_3'$ the restriction of $\wcL'$ and $\wcH_3'$ to $\wcX'_0 = \wX'$. By Theorem \ref{gor canonical K-ss limits}.(3) we know that $\wcH'_i$ are Cartier divisors on $\wcX'$. Let $h:\wX\dashrightarrow \wX'$ be the induced flop between the two different models of $X$.

\begin{prop}\label{semiampleness_diff_models}
    $\wH_3$ and $\wL'$ are both nef and semiample on $\wX$ and $\wX'$ respectively, and the Cartier divisors $\wcH_3$ and $\wcL'$ are $\wpi$-semiample and $\wpi$-big and $\wpi'$-semiample and $\wpi'$-big respectively.
\end{prop}

\begin{proof}
We begin by first showing that $\wH_3$ is nef on $\wX$. Recall that it is $g$-ample, hence for every curve $C$ in the $g$-exceptional locus we have $\wH_3\cdot C>0$. We now claim that the base locus of the linear series $|2\wH_3|$ is either some isolated points, or is contained in the $g$-exceptional locus, so $\wH_3$ is nef. By Lemma \ref{quasi polarised}, we know that $|2\wH_3|_{\wS}|$ is base-point-free, where $\wS\in|-K_{\wX}|$ is a general member. Suppose that $\wC\subseteq \Bs|2\wH_3|$ is a curve which is not contracted by $g$. Then the intersection $\wC\cap \wS$ is non-empty and consists of finitely many points, which are all base points of $|2\wH_3|_{\wS}|$ (cf. Lemma \ref{isomorphismonsection2}). This leads to a contradiction.

A similar argument applies to show that $\wL'$ is nef on $\wX'$, noting that the Lemmas \ref{quasi polarised} and \ref{isomorphismonsection2} apply directly due to the choice of polarisation and lattice in Section \ref{k33} to a general K3 surface $\wS'\in |-K_{\wX'}|$.

Since $\wL'= \wcL'|_{\wcX'_0}$ and $\wH_3= \wcH_3|_{\wcX_0}$ are both nef, and $\wcL'|_{\wcX'_t}$, $\wcH'_3|_{\wcX'_t}$ are nef for any $t\in T^{\circ}$ as $\wcX_t\simeq \wcX'_t\simeq \cX_t$ is a smooth Fano threefold in the family \textnumero 3.3, we conclude that and $\wcH_3$ is $\wpi$-nef and $\wcL'$ is $\wpi'$-nef. This implies the $\wpi$ and $\wpi'$-semiampleness of both respectively by Kawamata--Shokurov base-point-free theorem, as $\wcX$ and $\wcX'$ are of Fano type over $T$. Since $\wcL'|_{\wcX'_t}$, $\wcH'_3|_{\wcX'_t}$ are big for a general $t$, we get the $\wpi'$-bigness and $\wpi$-bigness respectively of each divisor.
\end{proof}

Since Proposition \ref{semiampleness_diff_models} shows that $\wcL'$ and $\wcH_3$ are $\wpi'$-semiample and $\wpi$-semiample respectively, we will take the ample models over $T$ for both of them. This yields birational morphisms $\phi_3\colon \wcX\rightarrow \cV_3$ and $\phi_{\wcL'}\colon \wcX'\rightarrow \cV_{\wL'}$ fitting into commutative diagrams 
\begin{center}
\begin{tikzcd}
\wcX \arrow[rr, "\phi_3"] \arrow[dr, "{\wpi}"'] && \cV_{3} \arrow[dl, "\wpi_{\cV_{3}}"] \\
& T
\end{tikzcd}  
\end{center}
and
\begin{center}
\begin{tikzcd}
\wcX' \arrow[rr, "\phi_{\wcL'}"] \arrow[dr, "{\wpi'}"'] && \cV_{\wL'} \arrow[dl, "\wpi_{\cV_{\wL'}}"] \\
& T
\end{tikzcd}  
\end{center}
where, in particular, 
$$\cV_{3} \coloneqq \operatorname{Proj}_T\bigg(\bigoplus_{m\in \N}\wpi_*\big(\wcH_3^{\otimes m}\big) \bigg)$$
and
$$\cV_{\wL'} \coloneqq \operatorname{Proj}_T\bigg(\bigoplus_{m\in \N}\wpi'_*\big(\wcL'^{\otimes m}\big) \bigg).$$

For any $t \in T^{\circ}$, the morphism $\wcX_t\rightarrow \big(\cV_{3}\big)_t$ is a projection from the smooth element of the threefold family \textnumero3.3 to $(\cV_3)_t\cong \PP^2$ which is a conic bundle whose discriminant curve is a smooth plane quartic curve. Meanwhile, the morphism $\wcX'_t\rightarrow \big(\cV_{\wL'}\big)_t$ is a morphism from the smooth element of threefold family \textnumero3.3 to $(\cV_{\wL'})_t\cong \PP^1\times \PP^1$ which is a conic bundle whose discriminant curve is of degree $(3,3)$. (c.f. \cite{CHELTSOV_FUJITA_KISHIMOTO_OKADA_2023} and \cite[\S 5.12]{acc2021}):

Let us consider the restriction of the morphisms $\phi_3$ and $\phi_{\wcL'}$ to the central fibre $(\phi_3)_0\colon \wcX_0\cong \wX\rightarrow V_3\coloneqq (\cV_{3})_0$ and $(\phi_{\wcL'})_0\colon \wcX_0\cong \wX'\rightarrow V_{\wL'}\coloneqq (\cV_{\wL'})_0$. Let $L_3\coloneqq \big((\phi_3)_0\big)_*  \wH_3$ and $L_{\wcL'}\coloneqq \big((\phi_{\wcL'})_0\big)_* ( \wL')$. The next lemmas and propositions study these morphisms. The first result is a direct consequence of the description of $V_3$ and $V_{\wL'}$.

\begin{lemma}\label{morphism is birational}
    The central fibres $V_3$ and $V_{\wL'}$  of $\wpi_{\cV_{3}}$ and $\wpi'_{\cV_{\wL'}}$ respectively are normal projective varieties. 
\end{lemma}

We will now study the properties of the above defined morphisms separately. We first prove some auxiliary lemmas.

\begin{lemma}\label{cohomology iso for H3}
   For any K3 surface $S \in |-K_{ \wX}|$, the restriction map $$H^0(  \wX,\OO_{ \wX}( \wH_3)) \rightarrow H^0(S,\OO_S ( \wH_3|_S ))$$ is an isomorphism. 
\end{lemma}
\begin{proof}
    Note that $\wH_3|_S$ is an ample line bundle on $S$, and 
    $$h^0(S, \wH_3|_S ) = \frac{1}{2} ( \wH_3|_S)^2 + 2 = 3.$$
    Consider the short exact sequence 
    \begin{center}
        \begin{tikzcd}
           0\arrow[r]& \OO_X( \wH_3-S)\arrow[r]& \OO_{ \wX}( \wH_3)\arrow[r]& \OO_S( \wH_3|_S)\arrow[r]& 0.
        \end{tikzcd}
    \end{center}
    Since $ \wH_3-S\sim - \wL$ is not effective, by taking the long exact sequence we see that $H^0( { \wX},\OO_{ \wX}( \wH_3)) \hookrightarrow H^0(S,\OO_S ( \wH_3|_S ))$ is injective, and thus $h^0( X, \OO_{ \wX}( \wH_3))\leq 3$. On the other hand, by upper semi-continuity, we have that 
    $$h^0(  X,\OO_{ \wX}( \wH_3)) \geq h^0( \wcX_t,\OO_{\wcX_t} ( \cH_t)) = 3$$ for a general $t \in T$. Therefore, one has $h^0(  X,\OO_{ \wX}( \wH_3))=3$, and the restriction map is an isomorphism.
\end{proof}

\begin{lemma}\label{cohomology iso for H1H2}
   For any K3 surface $S \in |-K_{ \wX'}|$, the restriction map $$H^0( X,\OO_X( \wL')) \rightarrow H^0(S,\OO_S (( \wL')|_S ))$$ is an isomorphism. 
\end{lemma}
\begin{proof}
    Follows the strategy of the proof of Lemma \ref{cohomology iso for H3} by noting that for a general K3 surface $S\in |-K_{ \wX'}|$, $ \wL'-S\sim - \wH_3'$ is not effective, and taking the same short exact sequence.
\end{proof}

\begin{prop}\label{studying H3}
    The variety $V_3$ is isomorphic to $\PP^2$, and the morphism $(\phi_3)_0$ is a conic bundle.
\end{prop}
\begin{proof}
By Lemma \ref{cohomology iso for H3}, $h^0(\wX,\wH_3)=3$, so $|\wH_3|$ defines a morphism (by taking non-zero sections of $H^0(\wX,\wH_3)$) $\Phi:\wX\to\PP^2$.
Notice that we have the Stein factorisation of $\Phi$: $\wX \xrightarrow{(\phi_3)_0} V_3 \xrightarrow{\nu} \PP^2$, with $V_3$ normal and $\nu$ finite. Let $S\in|-K_{\wX}|$ be a general K3 surface. Then $\Phi|_S=|\wH_3|_S|:S\to\PP^2$ has degree $2$ since $(\wH_3|_S)^2=2$.

A general $\Phi$–fibre $F$ intersects $\wL\cdot F =2$ and $\wH_3\cdot F=0$ (since these intersections are given on the generic fibre $\wcX_t$ and $\wcX\rightarrow T$ is flat)
so $(-K_{\wX})\cdot F=(\wL+\wH_3)\cdot F=2$. Thus $S\cdot F=2$, i.e. $\deg((\phi_3)_0|_S)=2$.
Consequently $2=\deg(\Phi|_S)=\deg((\phi_3)_0|_S)\deg(\nu)=2\deg(\nu)$, so $\deg(\nu)=1$. Therefore $\nu$ is an isomorphism and we have $V_3\simeq\PP^2$. As required.
In particular $L_3=\OO_{\PP^2}(1)$ and $(\phi_3)_0$ is a conic bundle.
\end{proof}

\begin{prop}\label{studying H1H2}
    The variety $V_{\wL'}$ is a canonical Gorenstein del Pezzo surface of degree $8$. In particular $V_{\wL'}\cong \PP^1\times \PP^1$ or $V_{\wL'}\cong \PP(1,1,2)$.
\end{prop}
\begin{proof}
Set $E:=\wpi'_*\cL'$ (rank $4$) and $\PP:=\PP(E^\vee)\xrightarrow{p}T$.
By vanishing and base change, $E$ and $\wpi'_*(2\cL')$ are vector bundles and formation
commutes with base change. The relative complete linear series gives a morphism $\Phi':\wcX'\to\PP$.
Now, we consider the multiplication map of vector bundles
$$\mu:\ \Sym^2 E\ \longrightarrow\ \wpi'_*(2\cL').$$
For $t\in T^\circ$ the fibre map $\mu_t$ has a $1$–dimensional kernel, corresponding to the unique quadric
in $\PP^3_t$ containing $\Phi'_t(\wX'_t)$. Since $T$ is a curve, these kernels glue to a line
subbundle $\mathcal{K}\subset\Sym^2 E$; after a finite base change trivialising $\mathcal{K}$,
this yields a global section $Q\in H^0(\PP,\OO_{\PP}(2))$ whose zero locus
$\mathcal{Q}:=(Q=0)\subset\PP$ is a relative irreducible quadric with
$\overline{\Phi'(\wcX')}\subset\mathcal{Q}$ and $\mathcal{Q}_t$ equal to the Segre quadric for $t\in T^\circ$. Irreducibility follows since it is the image of an irreducible variety under a morphism that is generically finite onto its image. In particular $\mathcal{Q}_t$ cannot be rank 2 (for all $t$).

Let $\wcX'\xrightarrow{\ \phi_{\wL'}\ }\cV_{\wL'}\xrightarrow{\ \nu\ }\PP$ be the relative Stein factorisation of $\Phi'$.
Then $\nu_t$ has degree $1$ on $T^\circ$. Moreover $(\nu)_*\OO_{\cV_{\wL'}}$ is torsion-free on the curve $T$,
hence locally free of rank $1$, so $\deg(\nu_t)$ is constant; thus $\deg(\nu_0)=\deg(\nu_t)=1$. This also shows that $\rank(\mathcal{Q}_0)\neq 1$, hence $\mathcal{Q}_0$ is normal.
Since $(\cV_{\wL'})_0$ is normal, $\nu_0$ is an isomorphism onto $\mathcal{Q}_0\subset\PP^3$.
Therefore $V_{\wL'}\cong\mathcal{Q}_0$ is a normal quadric surface, hence canonical and Gorenstein, since it is a hypersurface, i.e. a local complete intersection, in $\PP^3$. By adjunction, $-K_{V_{\wL'}}$ is ample with
$(-K_{V_{\wL'}})^2=8$, and $V_{\wL'}$ is either $\PP^1\times\PP^1$ or the quadric cone $\PP(1,1,2)$ (c.f. \cite[Theorem 3.4.(iv)]{Hidaka-Watanabe}).
\end{proof}

Recall that over $X$ we have small $\Q$-factorial modifications
$$
Y \dashrightarrow \wX \xrightarrow{\ (\phi_{3})_0\ } V_3\simeq \PP^2,
\qquad
Y \dashrightarrow \wX' \xrightarrow{\ (\phi_{\wcL'})_0\ } V_{\wL'}\in\{\PP^1\times\PP^1,\ \PP(1,1,2)\},
$$
where on $\wX$ (resp.\ $\wX'$) the divisor $\wHthree$ (resp.\ $\wLprime$) is semiample and defines $(\phi_3)_0$ (resp.\ $(\phi_{\wcL'})_0$). Let
$$
A_1\ :=\
\begin{cases}
\OO_{\PP^1\times\PP^1}(1,1), & V_{\wL'}=\PP^1\times\PP^1,\\
\OO_{\PP(1,1,2)}(2), & V_{\wL'}=\PP(1,1,2) \ (\text{so }A_1 \text{ is Cartier}).
\end{cases}
$$
On the big open $U\subset Y$ where all SQMs are isomorphisms, the rational maps induced by $|mL_Y|$ and $|nH_Y|$ agree with $(\phi_{\wcL'})_0$ and $(\phi_{3})_0$ via these identifications.

We define
$$
\Psi := ((\phi_{\wcL'})_0\circ\iota',\ (\phi_{3})_0\circ\iota):\ U\ \dashrightarrow\ V_{\wL'}\times \PP^2,
$$
where $\iota:U\hookrightarrow \wX$ and $\iota':U\hookrightarrow \wX'$ are the natural open immersions. Now, we set
$$
W\ :=\ \overline{\Psi(U)}\ \subset\ V_{\wL'}\times \PP^2 .
$$

We write $p_1:V_{\wL'}\times\PP^2\to V_{\wL'}$ and $p_2:V_{\wL'}\times\PP^2\to\PP^2$ for the natural projections, and we set
$$
\mathcal M \ :=\ p_1^*A_1\ \otimes\ p_2^*\OO_{\PP^2}(1).
$$
By construction, on $U$ we have
\begin{equation}\label{eq:pullbackM}
\Psi^*\mathcal M\ \cong\ L_Y|_U\ \otimes\ H_Y|_U\ \cong\ (-K_Y)|_U .
\end{equation}
Let $[W]\in \mathrm{Pic}(V_{\wL'}\times\PP^2)$ denote the class of the (Cartier) divisor $W$.

\begin{prop}\label{prop:classW}
With notation as above,
$$
[W]\ =\
\begin{cases}
\OO_{\PP^1\times\PP^1\times\PP^2}(1,1,2), & \text{if }V_{\wL'}=\PP^1\times\PP^1,\\[4pt]
\OO_{\PP(1,1,2)\times\PP^2}(2,2), & \text{if }V_{\wL'}=\PP(1,1,2).
\end{cases}
$$
In particular, $W$ is a prime hypersurface of multidegree $(1,1,2)$ or $(2,2)$.
\end{prop}

\begin{proof}
On $W_{\mathrm{reg}}$ we have $K_W+\mathcal M|_W\sim 0$.
Indeed, on the common big open $U\subset Y$ where all SQMs are isomorphisms, we have 
$\Psi^*\mathcal M\cong(-K_Y)|_U$ by \eqref{eq:pullbackM}, and $\Psi:U\to W$ is birational.
Hence the pullback of $K_W+\mathcal M|_W$ to $U$ is linearly trivial, which gives 
$K_W+\mathcal M|_W\sim 0$ on $W_{\mathrm{reg}}$.

We compute by adjunction on the product $B:=V_{\wL'}\times\PP^2$. By adjunction for the Cartier hypersurface $W\subset B$, on $W_{\mathrm{reg}}$ we have
$K_W\sim (K_B+[W])|_W$. Combining with the above yields
$$(K_B+[W])|_W \ \sim\ -\,\mathcal M|_W.$$
Let $H_1,H_2$ denote the two rulings on $\PP^1\times\PP^1$ (when applicable) and let $H_3$ denote the hyperplane class on $\PP^2$; on $\PP(1,1,2)$ let $H_1$ denote the generator with $\OO(1)$ $\Q$-Cartier and $\OO(2)$ Cartier.

\smallskip
\underline{\emph{Case ${V_{\wL'}=\PP^1\times\PP^1}$.}}
Then $K_{B}=-2H_1-2H_2-3H_3$. Write $[W]=aH_1+bH_2+cH_3$ with $a,b,c\in\mathbb{Z}$, so
$$
K_W \ =\ (K_B+[W])|_W \ =\ ((a{-}2)H_1+(b{-}2)H_2+(c{-}3)H_3)|_W .
$$
On the other hand, \eqref{eq:pullbackM} and the definition of $\mathcal M$ give
$$
-K_W\ =\ \mathcal M|_W\ =\ (H_1+H_2+H_3)|_W .
$$
Comparing coefficients yields $2-a=1$, $2-b=1$, $3-c=1$, i.e.\ $(a,b,c)=(1,1,2)$.

\smallskip

\underline{\emph{Case ${V_{\wL'}=\PP(1,1,2)}$.}}
Here $K_{V_{\wL'}}\sim \OO(-4)$, hence $K_B\sim -4H_1-3H_3$. Writing $[W]=d\,H_1+c\,H_3$ with $d,c\in\mathbb{Z}$, we get
$$
K_W \ =\ ((d{-}4)H_1+(c{-}3)H_3)|_W .
$$
Again $-K_W=\mathcal M|_W=(2H_1+H_3)|_W$ (since $A_1=\OO(2)$ and $p_2^*\OO(1)=H_3$), hence $4-d=2$ and $3-c=1$, i.e.\ $(d,c)=(2,2)$.
\end{proof}

The following are auxiliary Lemmas well known to experts.
\begin{lemma}\label{lem:sections}
Let $D$ be a Cartier divisor on $Y$. For any SQM $Y\dashrightarrow Y_1$ with strict transform $D_1$, the restriction induces isomorphisms
$$
H^0(Y,mD)\ \xrightarrow{\ \cong\ }\ H^0(Y_1,mD_1)\qquad(\forall\,m\ge0).
$$
\end{lemma}

\begin{proof}
A small modification is an isomorphism in codimension $\ge2$. Since $Y,Y_1$ are normal, $\OO_Y(mD)$ and $\OO_{Y_1}(mD_1)$ are reflexive and coincide over the common big open. Global sections of reflexive sheaves are determined on codimension $\ge2$ opens; hence the claim follows.
\end{proof}

\begin{lemma}\label{lem:factor}
Let $V:=\mathrm{Proj}\,\bigoplus_{k\ge0}H^0(Y,k(-K_Y))$ be the anticanonical model of $Y$. Then the anticanonical map $\phi_{|-K_Y|}:Y\to V$ factors as
$$
Y\ \dashrightarrow\ W\ \xrightarrow{\ |\mathcal M|_W|\ }\ V,
$$
and the induced map on graded rings
$$
\bigoplus_{k\ge0} H^0 \big(W,\ \mathcal M^{\otimes k}|_W\big)\ \longrightarrow\ \bigoplus_{k\ge0} H^0(Y,k(-K_Y))
$$
is an isomorphism in all sufficiently large degrees.
\end{lemma}

\begin{proof}
By construction, on $U\subset Y$ we have $\Psi^*\mathcal M\cong (-K_Y)|_U$ (eq.\ \eqref{eq:pullbackM}). Thus every section of $\mathcal M^{\otimes k}|_W$ pulls back to a section of $k(-K_Y)$ on $U$, hence on $Y$ since $Y$ is normal. Conversely, a section of $k(-K_Y)$ on $Y$ restricts on $U$ to a section of $((\phi_{\wcL'})_0^*A_1\otimes (\phi_{3})_0^*\OO(1))^{\otimes k}$, hence descends from a section of $\mathcal M^{\otimes k}|_W$. Kawamata--Viehweg vanishing on the product $V_{\wL'}\times\PP^2$ then shows that the restriction maps from the ambient to $W$ are surjective in all $k\gg0$, giving equality of graded pieces for $k\gg0$. Taking $\mathrm{Proj}$ yields the factorisation.
\end{proof}

\begin{prop}\label{prop:V=W}
The anticanonical model $V$ identifies with $W$; in particular, $W$ is normal.
\end{prop}

\begin{proof}
On $B=V_{\wL'}\times\PP^2$ the line bundle $\mathcal M:=p_1^*A_1\otimes p_2^*\OO_{\PP^2}(1)$ is very ample.
For $k\gg 0$, Serre vanishing on $B$ implies the restriction maps
$$H^0 \big(B,\mathcal M^{\otimes k}\big)\ \twoheadrightarrow\ H^0 \big(W,\mathcal M^{\otimes k}|_W\big)$$
are surjective. Hence $|\mathcal M^{\otimes k}|_W|$ is basepoint-free and defines a \emph{finite} morphism
$W\to \PP\big(H^0(W,\mathcal M^{\otimes k}|_W)^\vee\big)$.

By Lemma \ref{lem:factor}, the anticanonical map $\phi_{|-K_Y|}:Y\to V$ factors through this morphism
and induces a birational map $W\dashrightarrow V$. Therefore $W\to V$ is finite and birational; by
Zariski’s Main Theorem it is an isomorphism. In particular $W$ is normal.
\end{proof}

\begin{prop}\label{prop:X=W}
We have isomorphisms of anticanonical rings
$$
\bigoplus_{k\ge0} H^0(X,k(-K_X))\ \cong\ \bigoplus_{k\ge0} H^0(Y,k(-K_Y)),
$$
hence $\mathrm{Proj}\,\bigoplus_{k\ge0} H^0(X,k(-K_X)) \ \simeq\ V \ \simeq\ W$.
In particular, since $-K_X$ is ample, $X\simeq W$.
\end{prop}

\begin{proof}
Because $\theta:Y\to X$ is small and crepant, we have $\theta^*(-K_X)=-K_Y$ and the sections agree. By Lemma \ref{lem:sections} $H^0(X,k(-K_X))\cong H^0(Y,k(-K_Y))$ for all $k\ge0$. Thus, the rings coincide, and so do their $\mathrm{Proj}$'s. Finally, $-K_X$ is ample on the normal projective variety $X$, and therefore $X$ is isomorphic to the image of $\phi_{|-K_X|}$ and therefore to $\mathrm{Proj}\,\bigoplus_k H^0(X,k(-K_X))$.
\end{proof}

The following is then immediate.

\begin{cor}\label{cor:main}
$X$ is either embedded in $\PP^1\times\PP^1\times\PP^2$ as a $(1,1,2)$ divisor, or it is embedded in $\PP(1,1,2)\times\PP^2$ as a $(2,2)$ divisor.
\end{cor}

We will now briefly describe the full birational picture of $X$ depending on whether $V_{\wL'}$ is smooth or singular. We first study the separate case where $V_{\wL'} = \PP^1\times \PP^1$. Notice that for $\phi_{\wcL'}:\wcX\rightarrow \cV_{\wL'}$, there exists a line bundle $A$ on $\cV_{\wL'}$ whose restriction $A_t$ to each smooth fibre $(\cV_{\wL'})_t$ is one ruling on the quadric surface (so $\OO_{(\cV_{\wL'})_t}(1)\cong A_t\otimes B_t$ with $A_t,B_t$ the rulings). 

\begin{lemma}\label{lem:factor-semiample}
In the setting discussed above, the composition
$$
\Phi:=r\circ\phi_{\wcL'}: \wcX\longrightarrow\PP^1_T,\qquad r:=\Phi_{|A|}:\cV_{\wL'}\to\PP^1_T
$$
is a morphism over $T$ and
$$
\phi_{\wcL'}^*A=\Phi^*\OO_{\PP^1}(1)
$$
is $ \wpi$–semiample (hence $ \wpi$–nef).
In particular, if both ruling classes extend on $\cV_{\wL'}$, i.e. when $V_{\wL'} = \PP^1\times \PP^1$ is a smooth quadric, then both $\wcH_1'$ and $\wcH_2'$ are $ \wpi$–semiample and $ \wpi$–nef.
\end{lemma}

\begin{proof}
Since $A$ is relatively basepoint free on $\cV_{\wL'}/T$, it defines $r:\cV_{\wL'}\to\PP^1_T$ with $A=r^*\OO_{\PP^1}(1)$. Then $\Phi=r\circ\phi_{\wcL'}$ is a morphism and
$\Phi^*A=\phi^*\OO_{\PP^1}(1)$, which is relatively globally generated, hence $ \wpi$–semiample and nef.
\end{proof}

Thus we conclude that when $V_{\wL'} = \PP^1\times \PP^1$ the line bundles $\wcH_1'+\wcH_3$ and $\wcH_2'+\wcH_3$ are $ \wpi$-semiample. In particular, their ample models $\cV_{1,3}$ and $\cV_{2,3}$ define morphisms $\phi_{1,3}:\wcX\rightarrow \cV_{1,3}$ and $\phi_{2,3}:\wcX\rightarrow \cV_{2,3}$ such that for all $t\in T^{\circ}$ the restriction morphisms are birational from the smooth element of threefold family \textnumero3.3 to $(\cV_{13})_t\cong (\cV_{23})_t\cong \PP^1\times \PP^2$ which the blow up of $\PP^1\times \PP^2$ along a genus $3$ curve given as the complete intersection of two $(1,2)$ surfaces in $\PP^1\times \PP^2$. Denote by $V_{13}:=(\cV_{13})_0$ and $V_{23}:=(\cV_{23})_0$. Then, we also conclude the following.

\begin{lemma}\label{lem: auxialiary}
    Assume $V_{\wL'} = \PP^1\times \PP^1$ is smooth quadric.
    The morphisms $(\phi_{1,3})_0:X\rightarrow V_{1,3}$ and $(\phi_{2,3})_0:X\rightarrow V_{2,3}$ are birational and contract the exceptional divisors $E_1$ and $E_2$, where $E_1+E_2 = 4H_3$, to curves in $V_{1,3}$ and $V_{2,3}$, respectively.
\end{lemma}
\begin{proof}
    The morphisms are birational by a similar argument as in the proof of \cite[Lemma 4.11]{LZ24}. Consider the restriction map $\phi_{1,3}|_{\cE_1}:\cE_1\rightarrow \cC$. As $\cC$ is an irreducible scheme over $T$, whose general fibre is a curve, then $\cC$ is a surface, and $\cC_0$ is a curve. On the other hand, if $C'\subseteq X$ is a curve such that $C'$ is contracted by $(\phi_{1,3})_0$, then $C'\subseteq E_1$ because $2( H_1+ H_3)-(1-\epsilon)(  E_1+  E_2)$ is ample for $0<\epsilon \ll 1$. Hence, $ X$ is the blow up of a curve in $V_{1,3}$. The argument for $V_{2,3}$ is identical.
\end{proof}

To recap the above, $\cL:=\cH_1+\cH_2$ and $\cH_3$ are $ \pi$–semiample, with associated morphisms
$$
\phi_{\cL}:\cX\to \cV_{\cL},\qquad
\phi_{3}:\cX\to \PP^2_T,
$$
where $\cV_{\cL}:=\Proj_T\bigoplus_{m\ge0} \pi_*\OO_{\cX}(m\cL)$ is either a smooth quadric or a cone in $\PP^3$ and in particular $\cL=\phi_{\cL}^*\OO_{\mathcal S}(1)$, $\cH_3=\phi_{3}^*\OO_{\PP^2}(1)$. In particular, up to the possibility of $\cV_{\cL}$ we have the following two commutative diagrams on the central fibre $X = \cX_0$ 
\[
\begin{tikzcd}[column sep=5em, row sep=3em]
& \PP^1\times \PP^1 \arrow[dr] \arrow[dl] & \\
\PP^1 & & \PP^1 \\
& X 
  \arrow[uu, "(\phi_{\cL})_0" near start]
  \arrow[ul, "(\phi_{1})_0"', sloped, pos=0.6] 
  \arrow[ur, "(\phi_{2})_0", sloped, pos=0.6]
  \arrow[dl, "(\phi_{1,3})_0"', sloped, pos=0.4] 
  \arrow[dr, "(\phi_{2,3})_0", sloped, pos=0.4] 
  \arrow[dd, "(\phi_{3})_0" near start] & \\
\PP^1\times \PP^2 \arrow[uu] \arrow[dr] & & \PP^1\times \PP^2 \arrow[uu] \arrow[dl] \\
& \PP^2 &
\end{tikzcd}
\]
(when $V_{\cL} = \PP^1\times \PP^1$) and
\begin{center}
\begin{tikzcd}[column sep=huge, row sep=huge]
  & {X_0} \arrow[dl, "{(\phi_{\cL})_0}"'] \arrow[d, "{((\phi_{\cL})_0,(\phi_{3})_0)}"] \arrow[dr, "{(\phi_{3})_0}"] & \\
  {\PP(1,1,2)} & {\PP(1,1,2) \times \mathbb{P}^2} \arrow[l, "{pr_1}"'] \arrow[r, "{pr_2}"] & {\mathbb{P}^2}
\end{tikzcd}
\end{center}
(when $V_{\cL} = \PP(1,1,2)$), as the only possibilities. 

\section{K-moduli of family 3.3 and GIT}\label{sec:k-moduli}
In this section we will describe the K-moduli space of family \textnumero3.3 using the GIT of $(1,1,2)$ divisors in $\PP^1\times \PP^1\times \PP^2$.

\subsection{Deformation theory of K-semistable limits}\label{sec:def_theo}

\begin{prop}\label{deformation}
    Let $X$ be either a $(1,1,2)$ divisor in $\PP^1\times \PP^1\times \PP^2$ or a $(2,2)$ divisor in $\PP(1,1,2)\times \PP^2$. Then, there are no obstructions to deformations of $X$.
\end{prop}
\begin{proof}
    The statement is an easy application of \cite[Examples 3.2.11]{Sernesi2006} (itself derived from \cite{KodairaSpencer1958}) by noting that in both cases we have $H^1(X, N_{X/\PP}) = 0$ (where $\PP$ denotes the corresponding product of (weighted) projective spaces $X$ is embedded in) and $H^1(\PP, T_{\PP}|_X) = 0$.
\end{proof}
\begin{cor}\label{cor:stack-smooth}
    The K-moduli stack $\M^K_{\textup{№3.3}}$ is smooth, and its good moduli space $M^K_{\textup{№3.3}}$ is normal. Moreover,  $\M^K_{\textup{№3.3}}$ is a smooth connected component of $\M^K_{3, 18}$.
\end{cor}

\begin{proof}
For any point $[X]\in \M^K_{\textup{№3.3}}$, Corollary \ref{cor:main} implies that $X$ is Gorenstein canonical and isomorphic to either a $(1,1,2)$ divisor in $\PP^1\times \PP^1\times \PP^2$ or a $(2,2)$ divisor in $\PP(1,1,2)\times \PP^2$. Thus the ($\Q$-Gorenstein) deformation of $X$ is unobstructed by Proposition \ref{deformation}. Therefore, the K-moduli stack $\M^K_{3,18}$ is smooth at $[X]$, and hence $\M^K_{\textup{№3.3}}$ is a smooth connected component of $\M^K_{3,18}$. The normality of $M^K_{\textup{№3.3}}$ follows from smoothness of $\M^K_{\textup{№3.3}}$ by \cite[Theorem 4.16(viii)]{Alp13}. 
\end{proof}

\subsection{\texorpdfstring{GIT of (1,1,2) divisors in $\PP^1\times \PP^1\times \PP^2$}{GIT of (1,1,2) divisors in P1xP1xP2}}\label{sec: GIT1}
In this section, we will give a full description of the GIT of $(1,1,2)$ divisors in $\PP^1\times \PP^1\times \PP^2$. We will do so computationally, using the computer code \cite{GIT_code} and the methods developed in \cite{karagiorgis2023gitstabilitydivisorsproducts}. 

Let $G \coloneqq \SL(2) \times \SL(2)\times \SL(3)$, and $V \coloneqq |\OO_{\PP^1 \times\PP^1 \times\PP^2}(1,1,2)|$. We will study the GIT quotient $\mathbb{P} V^* \sslash G$ computationally.

Recall that we denote a normalised one-parameter subgroup $\lambda$ in $G$ by $$\lambda = (u,-u,v,-v,s_0,s_1, s_2).$$ By our algorithm, we obtain $S_{1,1,1}^{1,1,2}$ as in \cite[Definition 3.1, Theorem 3.2] {karagiorgis2023gitstabilitydivisorsproducts}. This set has 1563 elements, which we won't list here. The relevant one-parameter subgroups that give maximal semi-destabilised sets are 
\begin{align*}
\lambda_0 &= (0, 0, 0, 0, 1, 0, -1),\\
\lambda_1 &= (3, -3, 0, 0, 2, -1, -1),\\
\lambda_2 &= (1, -1, 1, -1, 0, 0, 0),\\
\lambda_3 &= (0, 0, 2, -2, 1, 0, -1),\\
\lambda_4 &= (1, -1, 1, -1, 1, 0, -1)\\
\lambda_5 &= (1, -1, 3, -3, 2, 0, -2)\\
\lambda_6 &= (1, -1, 1, -1, 2, 0, -2)\\
\lambda_7 &= (0, 0, 3, -3, 2, 2, -4)\\
\end{align*}

These give eight separate maximal semi-destabilising sets $N^{\oplus}(\lambda_i)$, which we will analyse in order to determine the stable, polystable and semistable orbits. Note that only the sets generated by $\lambda_0$, $\lambda_2$, $\lambda_4$ and $\lambda_5$ and $\lambda_6$ are strictly semistable. We first prove the following.

\begin{lemma}\label{singular 1,1,2 divisors}
    Let $X$ be a singular $(1,1,2)$ divisor in $\PP^1_\mathbf{x} \times\PP^1_{\mathbf{y}} \times\PP^2_\mathbf{z}$. Then $X = \bV(f)$ is given, up to projective equivalence, by 
    $$f = x_0y_0f_2(\mathbf{z})+x_1y_0g_2(\mathbf{z})+x_0y_1h_2(\mathbf{z})$$
    where $f_2$, $g_2$ and $h_2$ are homogeneous polynomials of degree $2$ in variables $\mathbf{z} = (z_0,z_1,z_2)$, or as a degeneration of the above.
\end{lemma}
\begin{proof}
Left to the reader.
\end{proof}

\begin{prop}\label{stable divisors}
    A $(1,1,2)$ divisor $X$ in $\PP^1_\mathbf{x} \times\PP^1_{\mathbf{y}} \times\PP^2_\mathbf{z}$ is GIT stable if and only if it is smooth.
\end{prop}
\begin{proof}
    Let $X$ be not stable. Then by \cite[Theorem 3.2, Theorem 3.4]{karagiorgis2023gitstabilitydivisorsproducts} $X$ is not stable if and only if $X=\bV(f)$ for some $f$ generated by the monomials of one of the eight separate maximal semi-destabilising sets $N^{\oplus}(\lambda_i)$. Consider $N^{\oplus}(\lambda_2)$. Then, 
    $$f = x_0y_0f_2(\mathbf{z})+x_1y_0g_2(\mathbf{z})+x_0y_1h_2(\mathbf{z})$$
    and by Lemma \ref{singular 1,1,2 divisors} we conclude that $X$ must be singular. Hence, $X$ is not stable if and only if $X$ is singular, and thus we obtain the proof.
\end{proof}

Now we will classify semistable and polystable elements.

\begin{prop}\label{semistable divisors}
   A $(1,1,2)$ divisor $X$ in $\PP^1_\mathbf{x} \times\PP^1_{\mathbf{y}} \times\PP^2_\mathbf{z}$ is GIT strictly semistable if and only if $X$ is singular and has
   \begin{enumerate}
       \item non-isolated singularities of multiplicity $2$, or
       \item 12 $A_1$ singulatities, or
       \item one $A_3$ singularity, or
       \item one $A_3$ and one $A_1$ singularity, or
       \item one $D_4$ singularity.
   \end{enumerate}
\end{prop}
\begin{proof}
Recall by \cite[Theorems 3.2, 3.4 and 4.6]{karagiorgis2023gitstabilitydivisorsproducts} that $X = \bV(f)$ is strictly semistable, if it is not stable, i.e. the monomials of $f$ belong to some $N^{\oplus}(\lambda_i)$, for some $\lambda_i\in S_{1,1,1}^{1,1,2}$, and it is strictly semistable by the centroid criterion. As mentioned before, only the sets generated by $\lambda_0$, $\lambda_2$, $\lambda_4$ and $\lambda_5$ and $\lambda_6$ are strictly semistable, so we analyse these to conclude. 

$N^{\oplus}(\lambda_0)$: Here $f$ is given by  $$f = \sum_{i,j\in \{0,1\}, 0\leq k\leq 3} x_{i}y_j(q_2^k(z_0,z_1)+(-1)^kz_0z_2),$$
where $q_2^i$ is a quadric form (this notation will be used throughout). $X$ is an irreducible singular threefold, with singularities at $[[a:b], [c:d], [0:0:1]]$, for arbitrary $a$, $b$, $c$, $d$. The singularities have multiplicity $2$ and are non-isolated.

$N^{\oplus}(\lambda_2)$: $f$ is given by $$f = x_0y_0f_2(\mathbf{z})+x_1y_0g_2(\mathbf{z})+x_0y_1h_2(\mathbf{z})$$ and $X$ is a singular irreducible threefold with $12$ $A_1$ singularities at points $[[0:1],[0:1],[a:b:c]]$, $[[0:1],[1:0],[a:b:c]]$ and $[[1:0],[0:1],[a:b:c]]$, where $[a:b:c]$ is a point of the intersections $f_2\cap g_2$, $h_2\cap g_2$ and $f_2\cap h_2$, respectively, giving $12$ singular points in total.

$N^{\oplus}(\lambda_4)$: $f$ is given by $$x_0y_0f_2(\mathbf{z})+\sum_{i\neq j}x_iy_j(q_2^1(z_0,z_1)+a_iz_0z_2)+bx_1y_1z_0^2,$$ 
and $X$ is a singular irreducible threefold with an $A_3$ singularity at the point $[[0:1],[0:1],[0:0:1]]$.

$N^{\oplus}(\lambda_5)$: $f$ is given by
$$x_0y_0f_2(\mathbf{z})+x_0y_1z_0l(z_0,z_1)+x_1y_0(q_2^1(z_0,z_1)+z_2l(z_0,z_1))+x_1y_1z_0^2,$$ 
where $l$ denotes a linear form, and $X$ is the irreducible singular threefold with one $A_3$ singularity at the point $[[0:1],[0:1],[0:0:1]$ and one $A_1$ singularity at the point $[[0:1],[0:1],[0:1:0]]$.

$N^{\oplus}(\lambda_6)$: $f$ is given by
$$x_0y_0(q_2^1(z_0,z_1)+z_2l(z_0,z_1))+ x_1y_0(q_2^2(z_0,z_1)+az_2z_0)+ x_0y_1(q_2^3(z_0,z_1)+bz_2z_0) +x_1y_1z_0l(z_0,z_1), $$
and $X$ is the irreducible singular threefold with a $D_4$ singularity at $[[1:0],[1:0],[0:0:1]]$.
\end{proof}

\begin{prop}\label{polystable divisors}
   A $(1,1,2)$ divisor $X$ in $\PP^1_\mathbf{x} \times\PP^1_{\mathbf{y}} \times\PP^2_\mathbf{z}$ is GIT strictly polystable if and only if $X$ is singular and 
   \begin{enumerate}
       \item has two non-isolated singularities of multiplicity $2$, or
       \item $X$ is isomorphic to the unique special non-reduced and reducible threefold $\wX = \bV(f)$ with $\Aut(\wX) = \SL(2)^3$, where 
       $$\tilde{f} = (z_1^2+z_0z_2)(x_0y_0 + x_1y_0+x_0y_1+x_1y_1)$$
       \item has 8 $A_1$ singularities, or
       \item has two $A_3$ singularities, or
       \item has two $A_3$ and two $A_1$ singularities, or
       \item has two $D_4$ singularities.
   \end{enumerate}
\end{prop}
\begin{proof}
    Recall by \cite[Proposition 3.5]{karagiorgis2023gitstabilitydivisorsproducts} that $X = \bV(f)$ is strictly polystable, if it is the monomials of $f$ belong to some $N^{0}(\lambda_i)$, for some $\lambda_i\in S_{1,1,1}^{1,1,2}$, such that the corresponding $N^{\oplus}(\lambda_i)$, corresponds to a strictly semistable $X$ by the centroid criterion. Thus, by the above discussion, we have to analyse only the sets $N^0(\lambda_i)$ generated by $\lambda_0$, $\lambda_2$, $\lambda_4$ and $\lambda_5$ and $\lambda_6$ to conclude.

    $N^{0}(\lambda_0)$: Here $f$ is given by  $$f = \sum_{i,j\in \{0,1\}, 0\leq k\leq 3} x_{i}y_j(a_1^kz_1^2+a_2^kz_0z_1).$$ In general $X$ is an irreducible singular threefold, with singularities at $[[a:b], [c:d], [0:0:1]]$ and $[[a:b], [c:d], [1:0:0]]$, for arbitrary $a$, $b$, $c$, $d$. The singularities have multiplicity $2$ and are non-isolated.

Notice, however, that there exists a special orbit, which after projective equivalence is given by $$\tilde{f} = (z_1^2+z_0z_2)(x_0y_0+x_1y_0+x_0y_1+x_1y_1).$$
This the unique non-reduced reducible threefold $\wX = \bV(\tilde{f})$ such that $\Aut(\wX) = \SL(2)^3$.

$N^{0}(\lambda_2)$: $f$ is given by $$f =x_1y_0g_2(\mathbf{z})+x_0y_1h_2(\mathbf{z})$$ and $X$ is a singular irreducible threefold with $8$ $A_1$ singularities at points $[[0:1],[0:1],[a:b:c]]$ and $[[1:0],[0:1],[a:b:c]]$, where $[a:b:c]$ is a point of the intersections $f_2\cap h_2$.

$N^{0}(\lambda_4)$: $f$ is given by $$ax_0y_0z_2^2+\sum_{i\neq j}x_iy_j(z_1^2+b_iz_0z_2)+cx_1y_1z_0^2,$$ 
and $X$ is a singular irreducible threefold with two $A_3$ singularities at the point $[[0:1],[0:1],[0:0:1]]$ and $[[0:1],[0:1],[1:0:0]]$.

$N^{0}(\lambda_5)$: $f$ is given by
$$ax_1y_1z_0^2 + bx_0y_1z_0z_1 + cx_1y_0z_1z_2 + dx_0y_0z_2^2$$ 
a, b, c, d are arbitrary coefficients, and $X$ is the irreducible singular threefold with two $A_3$ singularities at the points $[[0:1],[0:1],[0:0:1]]$ and $[[0:1],[0:1],[1:0:0]]$ and two $A_1$ singularities at the points $[[0:1],[0:1],[0:1:0]]$ and $[[1:0],[1:0],[0:1:0]]$.

$N^{0}(\lambda_6)$: $f$ is given by
$$ax_0y_0z_1z_2+ x_1y_0(z_1^2+bz_2z_0)+ x_0y_1(z_1^2+cz_2z_0) +dx_1y_1z_0z_1 $$
and $X$ is the irreducible singular threefold with two $D_4$ singularities at points $[[1:0],[1:0],[0:0:1]]$ and $[[1:0],[1:0],[1:0:0]]$.
\end{proof}

\subsection{\texorpdfstring{GIT of (2,2) divisors in $\PP(1,1,2)\times \PP^2$}{GIT of (2,2) divisors in P(1,1,2)xP2}}\label{sec: GIT2}
In this section, we will analyse the GIT of $(2,2)$ divisors in $\PP(1,1,2)\times \PP^2$ with no singular point at the vertex of $\PP(1,1,2)$.

Consider the $\SL(2)\times \SL(3)$-action on $H_{2,2}:=\PP H^0(\PP^1\times \PP^2,\mathcal{O}_{\PP^1\times \PP^2}(2,2))$ induced by the natural action of $\SL(2)\times \SL(3)$ on $\PP^1\times \PP^2$. We call $f_{2,2}\in H_{2,2}$ \emph{GIT-(semi/poly)stable} if it is a (semi/poly)stable point under the $\SL(2)\times \SL(3)$-action.

Before we proceed, we will construct the universal families for $(1,1,2)$ divisors in $\PP^1\times \PP^1\times \PP^2$ and $(2,2)$ divisors in $\PP(1,1,2)\times \PP^2$. For the first case, the universal family is given by $\pi\colon\PP^1\times \PP^1\times \PP^2\times H_{1,1,2}\rightarrow H_{1,1,2}$, where $H_{1,1,2}$ is the parameter scheme of $(1,1,2)$ divisors in $\PP^1\times \PP^1\times \PP^2$. Notice that the CM line bundle depends only on one parameter, and is ample by \cite{CP21,XZ21}.

For the latter case. Let $F_{2,2}$ be a normal $(2,2)$ divisor in $\PP(1,1,2)_{x,y,z}\times \PP^2_{w_0,w_1,w_2}$ given by the equation 
$$f_2(\mathbf{w})h_2(x,y)+ zg_2(\mathbf{w}) = 0,$$
where, since $F_{2,2}$ is not singular, $g_2(\mathbf{w})$ is not singular, and hence we may assume that $g_2(\mathbf{w}) = w_0w_1+w_2^2$.

Thus, we can identify $F_{2,2}$ to a point in $\A=H^0(\PP^1_{\mathbf{y}}\times \PP^2_\mathbf{z},\mathcal{O}_{\PP^1\times \PP^2}(2,2))$. Notice there is a natural $\mathbb{G}_m$-action on $\A$ by acting on the first factor of weight $(2,2)$.
Then $H_{2,2}$ is naturally isomorphic to the coarse moduli space of the DM stack $[(\A\setminus\{0\})/\mathbb{G}_m]$. There is a natural $\SL(2)\times \SL(3)$-action on $\A$ induced by the usual $\SL(2)\times \SL(3)$-action on $H^0(\PP^1\times \PP^2,\mathcal{O}(1,1))$. Since this $\SL(2)\times \SL(3)$-action commutes with the $\mathbb{G}_m$-action,
it descends to an $\SL(2)\times \SL(3)$-action on $H_{2,2}$. The above quotient has already been analysed in \cite[\S 3]{devleming2024kmodulispacefamilyconic}. We give the description of semistable and polystable orbits here for completion.

\begin{theorem}[{\cite[Proposition 3.7/Lemma 3.8]{devleming2024kmodulispacefamilyconic}}]\label{GIT ss for 2,2} For $b\in \C$, define the following bidegree $(2,2)$ polynomials:
\[
\begin{aligned}
    f_b^{(1)} & := y_0^2z_1^2+y_0y_1(bz_1^2+z_0z_2)+y_1^2z_1^2\\
    f_b^{(2)} & := y_0^2z_1z_2+y_0y_1(bz_1^2+z_0z_2)+y_1^2z_0z_1\\
    f_b^{(3)} &:= y_0^2z_2^2+y_0y_1(z_1^2+bz_0z_2)+y_1^2z_0^2
\end{aligned}
\]
For each $i=1,2,3$, let $R_b^{(i)}$ be the $(2,2)$-surface defined by $f_b^{(i)}$; then we have the following sets of GIT strictly semistable $(2,2)$-surfaces in $\PP^1_{\mathbf{y}}\times \PP^2_\mathbf{z}$:
$$
\begin{aligned}
    \mathcal{C}_{2\times A_1} & =\left\{\left. R_b^{(1)}\right\vert b\in \C\right\}/\langle b\sim -b\rangle \cup [R_0={y_0y_1(z_0z_2+z_1^2)=0}],\\
    \mathcal{C}_{\text{non-fin}} & =\left\{\left. R_b^{(2)}\right\vert b\in \C\right\}\cup [R_0],\\
    \mathcal{C}_{\text{non-red}} & =\left\{\left. R_b^{(3)}\right\vert b\in \C\right\}/\langle b\sim -b\rangle\cup [R_0].
\end{aligned}
$$
Furthermore, $R_b^{(2)}$ is not in the sme $SL_2\times SL_3$-orbit of $R_{b'}$ for any $b'\neq b$. For $i=1,3$, $R_b^{(i)}$ is in the same $SL_2\times SL_3$-orbit as $R_{b'}^{(i)}$ if and only if $b=\pm b'$.
\end{theorem}

\begin{theorem}[c.f. {\cite[Theorem 5.1]{devleming2024kmodulispacefamilyconic}}]\label{GIT ps for 2,2}
    The strictly polystable locus in the GIT moduli space $\overline{M}_{(2,2)}^{GIT}$ is the union of three rational curves $\mathcal{C}_{\text{non-fin}}\cup \mathcal{C}_{\text{non-red}}\cup \mathcal{C}_{2\times A_1} $. In particular, $R_3$ is GIT-polystable.
\end{theorem}

\subsection{Luna {\'e}tale slice and Kirwan blow up}

We first recall Luna's theorem for \'etale slices.

\begin{theorem}[Luna's Theorem]
    Let $X$ be an affine variety, $G$ a reductive group acting on $X$ and let $x\in X$ be a point with a closed orbit $G_x$. Then there exists a locally closed subvariety $V \subset X$ satisfying the following properties:
    \begin{enumerate}
        \item $V$ is an affine variety, with $x \in V$,
        \item $V$ is $G_x$-invariant,
        \item if $\psi:G_x\times V \rightarrow X$ is the $G$-morphism induced by the action of $G$ on $X$, then the image of $\psi$ is a saturated open set $U \subset X$,
        \item the map $\psi \big{|}_U$, the restriction of $\psi$ to $U$, is strongly \'etale.
    \end{enumerate}
    Moreover, we say that such a subvariety $V$ is an \emph{\'etale slice}. 
\end{theorem}
We now calculate \'etale slices for our case.

\begin{lemma}\label{lem: luna slice}
Let $\mathscr{X} \subset\PP^1_\mathbf{x}\times\PP^1_{\mathbf{y}}\times\PP^2_{\mathbf{z}}$ denote the non-reduced and reducible threefold which is given by the vanishing of the equation 
$$f=(x_0y_0+x_0y_1+x_1y_0+x_1y_1)(z_0z_2+z_1^2) = l(x,y)q(z).$$
Then a luna slice to $\left(\operatorname{SL}(2)\times\operatorname{SL}(2)\times\operatorname{SL}(3)\right)\cdot [\mathscr{X}]\subset\PP V^*\sslash G$ at $[\mathscr{X}]$ is given by the locally closed subset $W=:\{lq+g=0\}$, where $g\in H^0(\PP^1\times\PP^1\times\PP^2,\mathcal{O}(1,1,2))$ is a polynomial such that neither $l$ nor $q$ is a factor of $g$.
\end{lemma}
\begin{proof}
Let $X^{ss}$ denote the GIT semistable locus, and define $G:=\operatorname{SL}(2)\times\operatorname{SL}(2)\times\operatorname{SL}(3)$. As $\mathscr{X}$ is GIT-polystable, the orbit $G[\mathscr{X}]$ is closed inside $X^{ss}$. Define $V:=H^0(\PP^1,\mathcal{O}(1))$, then by \cite[Lemma 5.12]{ADL19}, we see that $$H^0(\PP^1\times\PP^1\times\PP^2,\mathcal{O}(1,1,2))=V \otimes V \otimes \left( \operatorname{Sym}^4V\oplus\C
   \right)$$
because $\operatorname{Sym}^0(V)=\C$. Note that the tangent space is given by $$T_{X^{ss},[f]}=\left(V\otimes V\otimes(\C\oplus\operatorname{Sym}^4(V))\right)/\C \cdot [f].$$ To determine the tangent space $T_{G [f],[f]}$ we consider small deformations $l_t$ of $l_0:=l$ inside $V\otimes V$ and $q_t$ of $q_0:=q$ inside $\C\oplus\operatorname{Sym}^4(V)$. Then we see that 
$$\frac{d}{dt}(l_tq_t)\big{|}_{t=0}=q\frac{dl}{dt}\big{|}_{t=0}+l\frac{dq}{dt}\big{|}_{t=0}$$
and therefore the tangent space to $G[f]$ at $[f]$ is given by $$T_{G[f],[f]}=\left(q\cdot (V\otimes V)+l\cdot\left(\C\oplus \operatorname{Sym}^4V\right)\right)/\C\cdot[f].$$
Therefore, the normal space at $[f]$ is given by the quotient
$$N_{G[f]/X^{ss},[f]}=\left(V\otimes V\otimes(\C\oplus\operatorname{Sym}^4(V))\right)/\left(q\cdot (V\otimes V)+l\cdot\left(\C\oplus \operatorname{Sym}^4V\right)\right),$$
applying the exponential map determines the luna slice 
$W=\{lq+g=0\}$, where $g\in H^0(\PP^1\times\PP^1\times\PP^2,\mathcal{O}(1,1,2))$ is a polynomial such that neither $l$ nor $q$ is a factor of $g$. 
\end{proof}

We now proceed with the characterisation of this Kirwan blow-up.

\begin{theorem}\label{thm: blow up on families}
Let $U_{1,1,2}\subset H_{1,1,2}$ be the GIT semistable locus under the action of $\PGL(2)\times \PGL(2) \times \PGL(3)$ and consider the universal family $$\left(\PP^1\times\PP^1\times\PP^2\times U_{1,1,2}\right)\rightarrow U_{1,1,2}.$$
After base change to a stacky weighted blowup $\widetilde{\mathcal{U}}\rightarrow U$ along the orbit $G\cdot [f]$ with stacky exceptional divisor $\mathcal{E}$ where $G=\operatorname{SL}(2)\times\operatorname{SL}(2)\times\operatorname{SL}(3)$ and a contraction morphism, there exists a family $\cY\rightarrow \widetilde{\mathcal{U}}$ which:
\begin{enumerate}
    \item is isomorphic to $\left(\PP^1\times\PP^1\times\PP^2\times U_{1,1,2}\setminus G\cdot[f]\right)\rightarrow U_{1,1,2}\setminus G\cdot [f]$ over the complement of $\mathcal{E}$,

    \item and whose fibres over $\mathcal{E}$ are $(2,2)$ divisors in $\PP(1,1,2)\times\PP^2$.

    \item Let $\mathcal{E}_W$ be the exceptional divisor of the stacky weighted blowup $\widetilde{\mathcal{W}}:=\widetilde{\mathcal{U}}\times_{U}W\rightarrow W$ over the luna slice $W$ in Lemma \ref{lem: luna slice}. Then the family $\cY\times_{\widetilde{\mathcal{U}}}\mathcal{E}_W$ over $\mathcal{E}_W$ is isomorphic to the universal family over $[(\mathbb{A}\setminus\{0\})/\mathbb{G}_m]$. 
\end{enumerate}
\end{theorem}
\begin{proof}
Recall that $N_{G[f]/X^{ss},[f]}$ that we got in Lemma \ref{lem: luna slice} can be rewritten as
    \[
    \begin{aligned}
    N_{G[f]/X^{ss},[f]} &=\left(V\otimes V\otimes(\C\oplus\operatorname{Sym}^4(V))\right)/\left(q\cdot (V\otimes V)+l\cdot\left(\C\oplus \operatorname{Sym}^4V\right)\right)\\
    &= \left((V\otimes V)/l\right)\otimes \left((\C\oplus\operatorname{Sym}^4(V))/q \right),
    \end{aligned}
    \]
    with an induced $G$-action. Let $\widetilde{\mathcal{U}}\to U$ be the stacky blowup along the $G$-orbit of $[f]$ and $\widetilde{U}\to U$ the scheme-theoretic blowup which are both $G$-equivariant. Let $\mathcal{E}$ denote the corresponding stacky exceptional divisor, and  scheme-theoretic exceptional divisors, respectively. Let 
    \[
    \pi\colon\mathcal{Z}_{\widetilde{U}}\rightarrow \widetilde{U}\quad \text{and}\quad\mathcal{Z}_{\widetilde{\mathcal{U}}}\rightarrow \widetilde{U}
    \]
    be the pullback of the universal family. Take the blow-up $\psi:\mathcal{X} \rightarrow \mathcal{Z}_{\widetilde{\mathcal{U}}}$ of the universal family along the component of a line and a conic of the fibre over $\mathcal{E}$, and let $\mathcal{G}$ be the exceptional divisor. Notice that this blow-up is $\SL(2)\times\SL(2)\times \SL(3)$-equivariant and also gives a flat family over $\widetilde{\mathcal{U}}$, whose fibres over points in $\mathcal{E}$ and $\widetilde{\mathcal{U}}\setminus \mathcal{E}$ are $X_{1,1,2}\cup X_{2,2}$  and $X_{1,1,2}$ respectively, where $X_{1,1,2}$ is a $(1,1,2)$ divisor in $\PP^1\times \PP^1\times \PP^2$, and, similarly $X_{2,2}$ is a $(2,2)$ divisor in $\PP(1,1,2)\times \PP^2$. In fact, the fibres over $\widetilde{\mathcal{U}}\setminus\mathcal{E}$ are not changed. 

    Let $\mathcal{H}$ be the proper transform of $\mathcal{Z}_{\mathcal{E}}\subseteq \mathcal{Z}_{\widetilde{\mathcal{U}}}$ in $\mathcal{X}$ and note that the pre-images of any point in $\mathcal{E}$ (respectively  $\mathcal{G}$) in $\mathcal{H}$ is $X_{1,1,2}$ (respectively $X_{2,2}$). Now we prove that $\mathcal{H}$ is contractible over $\widetilde{\mathcal{U}}$. Let $\mathcal{Q}_0$ be a divisor class on $\mathcal{X}$ obtained by pulling back the divisor $Q_0 = L+C$ in $\PP^1\times \PP^1\times \PP^2$, where $C$ is some conic in $\PP^2$ and $L$ is a generic $(1,1)$ divisor in $\PP^1\times \PP^1$. Then $\mathcal{Q}_0-\mathcal{G}$ is positive over $\widetilde{\mathcal{U}}\setminus\mathcal{E}$ and trivial over $\mathcal{E}$. Moreover, the restriction of $\mathcal{Q}_0-\mathcal{G}$ on the $X_{2,2}$ component of any fibre over $\mathcal{E}$ is $2f+2e$, where $f$ and $e$ are the two generators of the Picard group on $X_{2,2}$ respectively. In particular, we have that $(\mathcal{Q}_0-\mathcal{G}.f)=2$ and $(\mathcal{Q}_0-\mathcal{G}.e)=2$. As $\mathcal{Q}_0-\mathcal{G}$ is relatively nef over $\mathcal{E}$ and trivial on $\mathcal{H}$, then $a(\mathcal{Q}_0-\mathcal{G})-K_{\mathcal{Z}_{\widetilde{\mathcal{U}}}}$ is relatively ample over $\widetilde{\mathcal{U}}$ for some $a>0$. It follows from the base point free theorem that the divisor $\mathcal{Q}_0-\mathcal{G}$ gives a desired contraction over $\widetilde{\mathcal{U}}$. Denote by $\cY$ the image of the contraction. By our construction, the fibres in $\cY$ over points in $\widetilde{\mathcal{U}}\setminus \mathcal{E}$ and in $\mathcal{E}$ are $X_{1,1,2}$ and $X_{2,2}$ respectively. Moreover, as the line bundle $\mathcal{Q}_0-\mathcal{G}$ is naturally $\SL(2)\times \SL(2)\times \SL(3)$- polarised, the construction is $\SL(2)\times \SL(2)\times \SL(3)$-equivariant, and hence the construction descends to the quotient. This proves (1). 

    Take the standard degeneration of $\PP^1\times \PP^1\times \PP^2$ to $\PP(1,1,2)\times \PP^2$ in $\PP^3_{x,y,z,w} \times \PP^2_{z_0,z_1,z_2}$, given by the family $\cX_t=\{x^2+y^2+z^2=tw^2\}\rightarrow \A^1_t$, i.e. $\cX_t\times \PP^2\to  \A^1_t$. The fibre over $t\neq0$ is $\cX_t\simeq \PP^1_\mathbf{x}\times \PP^1_\mathbf{y}\times \PP^2_\mathbf{z}$ and the fibre over $0$ is $Q_0\times \PP^2$ where $Q_0$ is a quadric cone in $\PP^3$. If we consider the small resolution of singularities, the special fibre becomes $\PP(1,1,2)\times \PP^2$.
    
    Taking the pull-back of the equations on the Luna slice, one deduces that an equation over $t\neq 0$ is of the form (after a change of coordinates)    $$(l_{1,1}(x,y,z,w)(z_0z_1+z_2^2)+f_{1,1,2}(x,y,z,w,z_0,z_1,z_2)=0$$ and the equation over $t=0$ is $$s (z_0z_1+z_2^2)+(u^2+v^2+uv)(z_0z_1+z_2^2)+ g_{2,2}(u,v,z_0,z_1,z_2),$$
    i.e.
    $$s(z_0z_1+z_2^2)+ g_{2,2}(u,v,z_0,z_1,z_2),$$
    proving (2). 

    As the assigned weight on $\mathcal{N}_{G\cdot[Z]/U}$ is the same as that of the $\mathbb{G}_m$-action $\sigma$ on $\mathbb{A}$, we deduce (3) as desired.
    
\end{proof}

\subsection{K-moduli as global GIT quotients}\label{sec: iso}

We will now describe the K-moduli space of Fano threefold famaily \textnumero3.3

Let $$\pi:\PP(1,1,2)\times \PP^2\times \A \longrightarrow \A$$ be the universal family over $\A$. Here the fibre of $\pi$ over each point $f_{2,2}$ of $\A$ is $F_{2,2}$, where $F_{2,2}$ is as before. Then the $\mathbb{G}_m$-action on $\A$ naturally lifts to the universal family. After taking the $\mathbb{G}_m$-quotient, we obtain a $\Q$-Gorenstein
family of Fano varieties over $[(\A\setminus\{0\})/\mathbb{G}_m]$, and the CM $\Q$-line bundle $\lambda_{\CM,\pi}$ on $\A$ also descends to a $\Q$-line bundle on $H_{2,2}$, denoted by $\Lambda_{\CM}$, which is again ample by \cite{CP21,XZ21}.

\begin{theorem}\label{thm:K-ps iff GIT ps}
Let $X$ be a K-polystable variety in $M^K_{\textup{№3.3}}$. Then 
\begin{enumerate}
    \item either $X$ is a GIT-polystable $(1,1,2)$ divisor in $\PP^1\times \PP^1\times \PP^2$, not projectively isomorphic to $[\widetilde{X}]$;
    \item or $X$ is a GIT-polystable $(2,2)$ divisor in $\PP(1,1,2)\times \PP^2$. 
\end{enumerate} 
Conversely, any such variety $X$ is K-polystable.
\end{theorem} 
\begin{proof}
    We will start by showing that K-polystability implies GIT polystability. By Corollary \ref{cor:main} we know that $X$ is either a $(1,1,2)$ divisor in $\PP^1\times \PP^1\times \PP^2$ or a $(2,2)$ divisor in $\PP(1,1,2)\times \PP^2$. We will analyse (1) first. Consider the universal family $\pi\colon\PP^1\times \PP^1\times \PP^2\times H_{1,1,2}\rightarrow H_{1,1,2}$, where $H_{1,1,2}$ is the parameter scheme of $(1,1,2)$ divisors in $\PP^1\times \PP^1\times \PP^2$. By \cite[Theorem 2.22]{ADL19} we have that $X$ is GIT polystable if it is K-polystable. Similarly, for (2), by considering the universal family 
    $$\pi:\PP(1,1,2)\times \PP^2\times \A \longrightarrow \A,$$
    \cite[Theorem 2.22]{ADL19} combined with the discussion above shows that $X$ is GIT polystable if it is K-polystable.

    For the reverse direction, we analyse (1) first. Let $[X_{1,1,2}]$ be a GIT semistable point where $X_{1,1,2}$ is not as in Proposition \ref{polystable divisors}.(2), i.e. $[X_{1,1,2}]\in U$. Take $\{X_t\}_{t\in U}$ a family of $(1,1,2)$ divisors over a smooth pointed curve $(0\in T)$ such that $X_0=X_{1,1,2}$, $X_t$ is smooth and hence K-stable for any $t \in T \setminus \{0\}$ by \cite{CHELTSOV_FUJITA_KISHIMOTO_OKADA_2023}. Then by the properness of K-moduli spaces, we have a K-polystable limit $X'_0$ of $X_t$ as $t \in 0$, after a possible finite base change of $T$. By Corollary \ref{cor:main}, we know that $X'_0$ is either isomorphic to a $(1,1,2)$ divisor in $\PP^1\times \PP^1\times \PP^2$, or a $(2,2)$ divisor in $\PP(1,1,2)\times \PP^2$. By the separatedness of GIT quotients, we know that $X_{1,1,2}$ specially degenerates to $g\cdot X'_0$ for some $g \in\PGL(2)\times \PGL(2)\times \PGL(3)$. Thus $X_{1,1,2}$ is K-semistable by the openness of K-semistability. If in addition $X_{1,1,2}$ is GIT polystable, then $X_{1,1,2}$ has a K-polystable limit $X_0'$. In particular, we know that $X_0'$ is a GIT-polystable point which is S-equivalent to $X_{1,1,2}$. Hence $X_{1,1,2} = g \cdot X'_0$ for some $g\in\PGL(2)\times \PGL(2)\times \PGL(3)$ and $X_{1,1,2}$ is K-polystable.

    Now we analyse (2). Let  $F_{2,2}$ be a GIT-semistable divisor. We can take a family of divisors $(F_{2,2,b})_{b\in B}$ over a smooth pointed curve $0\in B$ such that $F_{2,2,0}\simeq F_{2,2}$ and that $F_{2,2,b}$ is a general $(2,2)$ divisor in $\PP(1,1,2)\times \PP^2$. By the properness of K-moduli spaces, one can find a limit $F'_{2,2,0}$ as $b\to 0$ such that $F'_{2,2,0}$ is K-polystable, after a finite base change. Then $F'_{2,2,0}$ is GIT polystable, and by the separatedness of GIT moduli spaces, we deduce that $F_{2,2}$ specially degenerates to $g\cdot (F'_{2,2,0})$ for some $g\in \PGL(2)\times \PGL(3)$. By openness of K-semistability, one has that $F_{2,2}$ is K-semistable. If moreover $F_{2,2}$ is  GIT polystable, then we must have that $F_{2,2}=g\cdot F'_{2,2,0}$ for some $g\in \PGL(2)\times \PGL(3)$, and hence $F_{2,2}$ is K-polystable.
\end{proof}

Let $\mathcal{I}\subseteq \mathcal{O}_{U}$ be the ideal sheaf such that $\widetilde{U}=\Bl_{\mathcal{I}}U$, and $\overline{\mathcal{I}}_t\subseteq \mathcal{O}_{H_{1,1,2}}$ be the $\SL(2)\times \SL(2)\times \SL(3)$-equivariantly extended ideal sheaf of $\mathcal{I}$ whose cosupport is the closure of $G\cdot [\widetilde{X}]$ in $H_{1,1,2}$. In fact, $\overline{\mathcal{I}}_t$ is the ideal sheaf of the smooth locus $G\cdot [\widetilde{X}$. Let $\pi_{H}:\widetilde{H}=\Bl_{\overline{\mathcal{I}}}(H_{1,1,2})\rightarrow H_{1,1,2}$ be the blow-up and $\overline{E}$ the exceptional divisor. Then the line bundle $L_{k}:=\pi_H^{*}\mathcal{O}(k)\otimes\mathcal{O}(-E)$ is an $\SL(2)\times \SL(2)\times \SL(3)$-linearised polarization on $\widetilde{H}$ for any sufficiently divisible and large integer $k$. It follows from \cite{Kir85} that the GIT-stability of $(\widetilde{H},\mathcal{L}_{k})$ is independent of the choice of $k\gg1$, and the GIT-semistable locus $\widetilde{H}^{ss}$ is contained in $\widetilde{U}=\pi_{H}^{-1}(U)$. Let $U^{ps}$ and $\widetilde{H}^{ps}$ be the GIT-polystable loci in $H_{1,1,2}$ and $\widetilde{H}$ respectively. Set $\widetilde{\mathcal{U}}^{ss}:=\widetilde{\mathcal{U}}\times_{\widetilde{U}}\widetilde{H}^{ss}$, and $\widetilde{\mathcal{U}}^{ps}:=\widetilde{\mathcal{U}}\times_{\widetilde{U}}\widetilde{H}^{ps}$. Let $\mathcal{M}^K_{3.3}$ be the K-moduli stack parametrising K-semistable Fano varieties in Fano threefold family \textnumero 3.3, with good moduli space ${M}^K_{3.3}$.

\begin{theorem}\label{thm: iso of stacks}
    There exists an isomorphism $$\psi:\left[\widetilde{\mathcal{U}}^{ss}/\PGL(2)\times \PGL(2)\times \PGL(3)\right]\longrightarrow \mathcal{M}^K_{3.3}.$$
\end{theorem}

\begin{proof}
    Let $E^{ps}_{V}$ be the GIT-polystable locus in the exceptional divisor $E_V$ of the blow-up $\widetilde{V}\rightarrow V$. Then by \cite{Kir85}, we have that $$\widetilde{U}^{ps}=\pi_{H}^{-1}\left(U^{ps}\setminus G\cdot[\widetilde{X}]\right)\cup G\cdot E^{ps}_{W}.$$ By Theorems \ref{thm: blow up on families} and \ref{thm:K-ps iff GIT ps}, the fibres of $\cY\rightarrow \widetilde{\mathcal{U}}$ over $\widetilde{U}^{ps}$ are all K-polystable. By openness, we deduce that each fibre over $\widetilde{U}^{ss}$ is K-semistable. The existence of morphism $\psi$ follows from the universal property of the K-moduli stacks.

    The proof of isomorphism between stacks is the same as that of \cite[Theorem 5.15]{ADL19}, which makes use of Alper's criterion (cf. \cite[Proposition 6.4]{Alp13}).
\end{proof}

\subsection{Classification of K-(semi/poly)stable objects}

Combining the results of Sections \ref{sec: GIT1} and \ref{sec: iso} as well as the GIT classification of \cite[\S 3]{devleming2024kmodulispacefamilyconic} which we summarise in Section \ref{sec: GIT2} and Theorems \ref{GIT ss for 2,2} and \ref{GIT ps for 2,2} we obtain the following explicit description of K-semistable, polystable and stable Fano varieties in family \textnumero3.3, which we include here for completion.

\begin{theorem}\label{thm: full K-ss description}
 A Fano threefold in family \textnumero3.3 is
    \begin{enumerate}
        \item K-stable if and only if it is smooth;
        \item strictly K-semistable if and only if $X$ is either a $(1,1,2)$ divisor in $\PP^1\times \PP^1\times\PP^2$ in which case $X$ has either
        \begin{enumerate}
            \item non-isolated singularities of multiplicity $2$, or
       \item 12 $A_1$ singulatities, or
       \item one $A_3$ singularity, or
       \item one $A_3$ and one $A_1$ singularity, or
       \item one $D_4$ singularity.
        \end{enumerate}
        or a $(2,2)$ divisor in $\PP(1,1,2)_{u,v,s}\times\PP^2_{\mathbf{w}}$ in which case $X = V(f)$, where $f = s(w_0w_1+w_2^2)+g_{2,2}(u,v,\mathbf{w})$, such that $Y = V(g_{2,2})$ is a $(2,2)$ divisor in $\PP^1\times \PP^2$ which has 
        \begin{enumerate}
           \item non isolated singularities along the point at infinity of $\PP^1$ and is reducible, or
           \item two $A_1$ singularities, or
           \item one $A_2$ singularity, or 
           \item one $A_3$ singularity;
        \end{enumerate}
        \item strictly K-polystable if and only if 
        if $X$ is either a $(1,1,2)$ divisor in $\PP^1\times \PP^1\times\PP^2$ in which case $X$ has either
        \begin{enumerate}
           \item two non-isolated singularities of multiplicity $2$, or
       \item 8 $A_1$ singularities, or
       \item two $A_3$ singularities, or
       \item two $A_3$ and two $A_1$ singularities, or
       \item two $D_4$ singularities.
        \end{enumerate}
        or a $(2,2)$ divisor in $\PP(1,1,2)_{u,v,s}\times\PP^2_{\mathbf{w}}$ in which case $X = V(f)$, where $f = s(w_0w_1+w_2^2)+g_{2,2}(u,v,\mathbf{w})$, such that $Y = V(g_{2,2})$ is a $(2,2)$ divisor in $\PP^1\times \PP^2$ which has  
        \begin{enumerate}
           \item non isolated singularities along along the points at zero and at infinity of $\PP^1$ and is reducible, or
           \item four $A_1$ singularities, or
           \item two $A_2$ singularities, or 
           \item two $A_3$ singularities.
        \end{enumerate}
    \end{enumerate}
\end{theorem}

\printbibliography
\end{document}